\documentclass[reqno]{amsart}
\usepackage{esint}
\usepackage{cite}
\usepackage{amssymb}
\usepackage[usenames, dvipsnames]{color}
\usepackage{marginnote}
\usepackage{todonotes}
\usepackage{hyperref}

\usepackage{verbatim}
\usepackage{color}
\newcommand{\RN}[1]{%
  \textup{\uppercase\expandafter{\romannumeral#1}}%
}
\DeclareMathOperator*{\osc}{\text{osc}}

\theoremstyle{plain}
\newtheorem{theorem}{Theorem}[section]
\newtheorem{lemma}[theorem]{Lemma}
\newtheorem{corollary}[theorem]{Corollary}
\newtheorem{proposition}[theorem]{Proposition}

\theoremstyle{definition}

\theoremstyle{remark}
\newtheorem{remark}[theorem]{Remark}

\providecommand{\keywords}[1]{\textbf{Key words.} #1}
\renewcommand{\div}{\text{div}}

\renewcommand{\epsilon}{\varepsilon}
\newcommand{\R}{\mathbb{R}}
\numberwithin{equation}{section}

\makeatletter
\@namedef{subjclassname@2020}{%
  \textup{2020} Mathematics Subject Classification}
\makeatother

\begin{document}
\title[Gradient estimates for parabolic $p$-Laplace type equations]{Gradient estimates for singular parabolic $p$-Laplace type equations with measure data}

\author[H. Dong]{Hongjie Dong}
\address[H. Dong]{Division of Applied Mathematics, Brown University, 182 George Street, Providence, RI 02912, USA}

\email{Hongjie\_Dong@brown.edu}

\thanks{H. Dong was partially supported by a Simons fellowship grant no. 007638, the NSF under agreement DMS-2055244, and the Charles Simonyi Endowment at the Institute of Advanced Study.}

\author[H. Zhu]{Hanye Zhu}
\address[H. Zhu]{Division of Applied Mathematics, Brown University, 182 George Street, Providence, RI 02912, USA}
\thanks{H. Zhu was partially supported by the NSF under agreement DMS-2055244.}

\email{Hanye\_Zhu@brown.edu}



\begin{abstract}
We are concerned with gradient estimates for solutions to a class of singular quasilinear parabolic equations with measure data, whose prototype is given by the parabolic $p$-Laplace equation $u_t-\Delta_p u=\mu$ with $p\in (1,2)$. The case when $p\in \big(2-\frac{1}{n+1},2\big)$ were
studied in \cite{kuusi2013mingione}. In this paper, we extend the results in \cite{kuusi2013mingione} to the open case when $p\in \big(\frac{2n}{n+1},2-\frac{1}{n+1}\big]$ if $n\geq 2$ and $p\in(\frac{5}{4}, \frac{3}{2}]$ if $n=1$. More specifically,  in a more singular range of $p$ as above, we establish
pointwise gradient estimates via linear parabolic Riesz potential and gradient continuity results via certain assumptions on parabolic Riesz potential.
\end{abstract}
\maketitle

\section{Introduction}
In this paper, we consider the quasilinear parabolic equation with measure data
\begin{equation}\label{eq:u}
   u_t -\div(a(x,t,Du))=\mu
\end{equation}
in a cylindrical domain $\Omega_T=\Omega\times(-T,0)\subset \R^n$, where $\Omega\subset\mathbb{R}^n$ is a bounded domain and $T>0$. Here and in what follows, the operators ``$D$'' and ``$\text{div}$'' stand for the gradient and divergence with respect to the space variable $x$. Moreover, $\mu$ is a finite signed Radon measure in $\Omega_T$, namely, $|\mu|(\Omega_T)<\infty$. The vector field $a=(a_1,\ldots,a_n):\Omega_T\times \R^n\rightarrow \R^n$ is assumed to satisfy the following growth, ellipticity, and continuity conditions: there exist constants $0<\nu\leq L$, $s\geq 0$, and $p>1$ such that
\begin{equation}\label{ineq:growth}
    |a(x,t,\xi)|+(s^2+|\xi|^2)^{1/2} |D_\xi a(x,t,\xi)|\leq L(s^2+|\xi|^2)^{(p-1)/2},
\end{equation}
\begin{equation}\label{ineq:elliptic}
    \left\langle D_\xi a(x,t,\xi)\eta,\eta \right\rangle\geq\nu(s^2+|\xi|^2)^{(p-2)/2}|\eta|^2,
\end{equation}
and
\begin{equation}\label{ineq:osi}
    |a(x,t,\xi)-a(x_0,t,\xi)|\leq L\,\omega(|x-x_0|)(s^2+|\xi|^2)^{(p-1)/2}
\end{equation}
hold for every $x,x_0\in\Omega$, $t\in(-T,0)$, and $(\xi,\eta)\in\mathbb{R}^n\times\mathbb{R}^n\backslash\{(0,0)\}$, where $\omega:[0,\infty)\rightarrow[0,1]$ is a
concave non-decreasing function satisfying
$$
\lim_{\rho\rightarrow0^+}\omega(r)=\omega(0)=0
$$
and the Dini condition
\begin{equation}\label{dini}
    \int_0^1\omega(\rho)\,\frac{d\rho}{\rho}<+\infty.
\end{equation}
A typical model equation is given by the (possibly nondegenerate) parabolic $p$-Laplace equation with measure data and $s\geq 0$:
\begin{equation*}
    u_t-\div\left((|D u|^2+s^2)^\frac{p-2}{2}D u\right)=\mu \quad \text{in} \,\, \Omega_T.
\end{equation*}
By a (weak) solution to the equation \eqref{eq:u}, we mean a function
$$
u\in C^0(-T,0;\, L^2(\Omega))\cap L^p(-T,0;\,W^{1,p}(\Omega))
$$
such that
the distributional relation
$$-\int_{\Omega_T} u\varphi_t\,dxdt+\int_{\Omega_T}\langle a(x,t,Du),D\varphi\rangle\, dxdt=\int_{\Omega_T}\varphi \,d\mu$$
holds whenever $\varphi\in C^\infty_0(\Omega_T)$ has compact support in $\Omega_T$.

The gradient estimates for the super-quadratic case when $p\geq 2$ were well studied in the literature. See \cite{MR2823872, kuusi2014the, kuusi2014riesz} and also
\cite{MR2729305, kuusi2014guide, MR2746772, MR3004772} for estimates for elliptic problems. However, the corresponding results for the singular case when $p\in(1,2)$ are still not complete.

In this paper, we are concerned with only the singular case when $p\in (1,2)$.

\subsection{Pointwise gradient estimates}

First, we recall the potential estimates of the gradients of solutions to the stationary equations
\begin{equation}\label{eq:elliptic1}
    -\text{div} \,a(x,Du)=\mu \quad \text{in}\,\,\Omega.
\end{equation}
The following pointwise gradient estimates were established in \cite{duzaar2010gradient} by Duzaar and Mingione for the case when $p\in(2-\frac{1}{n},2]$:
\begin{align*}
    |D u(x)|\leq  c\,\big[\mathbf{I}_1^{|\mu|}(x,R)\big]^\frac{1}{p-1}+c\fint_{B_R(x)}(|D u(y)|+s)\,dy
\end{align*}
holds for any solution $u$ to the equation \eqref{eq:elliptic1} and any ball $B_R(x)\subset\Omega$. Here $B_R(x)\subset \R^n$ denotes the ball centered at $x$ with radius $R$,
$\fint_E$ stands for the integral average over a measurable set $E$, and
\begin{equation}\label{riesz1}
    \mathbf{I}_1^{|\mu|}(x,R):=\int_0^R\frac{|\mu|(B_\rho(x))}{\rho^{n-1}}\,\frac{d\rho}{\rho}
\end{equation}
is the truncated first-order elliptic Riesz potential.
In \cite{dong2021gradient}, we extended the results above to include the case when $p\in \big(\frac{3n-2}{2n-1},2-\frac{1}{n}\big]$ and also derived the following Lipschitz estimates for the case when $p\in \big(1,\frac{3n-2}{2n-1}\big]$:
$$
   \|D u\|_{L^\infty(B_{R/2}(x))}
         \leq c \,\big\|\mathbf{I}_1^{|\mu|}(\cdot,R)\big\|^\frac{1}{p-1}_{L^\infty(B_{R}(x))}+c\, R^{-\frac{n}{2-p}} \||D u|+s\|_{L^{2-p}(B_{R}(x))}.
$$
For more gradient estimates for the elliptic problem in the singular case $p\in(1,2)$, we refer the reader to \cite{nguyen2019good,nguyen2020existence,nguyen2020pointwise}.
 The first gradient potential result for singular parabolic $p$-Laplace type equations was obtained by Kuusi and Mingione in \cite{kuusi2013mingione} for the case when $p\in (2-\frac{1}{n+1},2]$ using intrinsic geometry and exit time arguments. More precisely, they first showed that there exists a constant $c=c(n,p,\nu,L,\omega)$ such that if
$$
c \,\fint_{Q_{r_\lambda}^\lambda(x_0,t_0)} (|Du|+s)\,dxdt +c\, \int_{0}^{2r_\lambda} \frac{|\mu|(Q_\rho^\lambda(x_0,t_0)}{\rho^{n+1}} \frac{d\rho}{\rho}\leq \lambda
$$
for some constant $\lambda>0$, then $|Du(x_0,t_0)|\leq \lambda$.
Here $r_\lambda:=\lambda^{(p-2)/2}r$ and
\begin{equation}\label{eq:ic}
    Q_\rho^\lambda(x_0,t_0):=B_{\rho}(x_0)\times (t_0-\lambda^{2-p} \rho^2, t_0)
\end{equation}
is called an intrinsic cylinder for $\rho,\,\lambda>0$.
They also used the intrinsic Riesz potential result above to establish the following parabolic Riesz potential bound when $p\in (2-\frac{1}{n+1},2]$:
\begin{align*}
    &|Du(x_0,t_0)|\leq c\, [\mathbf{I}_1^{|\mu|}(x_0,t_0,2r)]^{2/[(n+1)p-2n]}
    \\&\quad+c\, \Big(\fint_{Q_r(x_0,t_0)}(|Du|+s+1)\,dxdt\Big)^{2/[2-n(2-p)]}
\end{align*}
holds for any solution $u$ to the equation \eqref{eq:u} in the standard parabolic cylinder
$
    Q_{2r}(x_0,t_0):=B_{2r}(x_0)\times (t_0-4r^2, t_0)\subset \Omega_T.
$
Here
\begin{equation}\label{riesz2}
\mathbf{I}_1^{|\mu|}(x_0,t_0;r):=\int_0^r\frac{|\mu|(Q_\rho(x_0,t_0))}{\rho^{n+1}}\,\frac{d\rho}{\rho}
\end{equation}
is the truncated first-order parabolic Riesz potential. For more notation in parabolic (intrinsic) geometry, see Section \ref{sec2.1} below.

In this paper, we extend their results to include the case when $p\in(p^*(n),2-\frac{1}{n+1}]$, where
\begin{equation}\label{pstar}
    p^*(n):=
    \max\Big\{\frac{2n}{n+1},\frac{3n+2}{2n+2}\Big\}=
    \left\{
    \begin{aligned}
    &\,\,\,\frac{5}{4} \quad &\text{when} \quad n=1,
    \\&\frac{2n}{n+1}  &\text{when}  \quad n\geq 2.
    \end{aligned}
    \right.
\end{equation}
Note that $p^*(n)<2-\frac{1}{n+1}$ holds for every integer $n\geq 1$. Moreover, it is clear that if $p\in(p^*(n),2-\frac{1}{n+1}]$, then
$$0<\max\big\{\frac{n+2}{2(n+1)}, \frac{(2-p)n}{2}\big\}< p-\frac{n}{n+1}\leq 1$$
so that we can choose a constant $q\in(0,1)$ satisfying
\begin{equation}\label{rangeq}
    q\in \Big(\max\big\{\frac{n+2}{2(n+1)}, \frac{(2-p)n}{2}\big\},\, p-\frac{n}{n+1}\Big)\subset (0,1).
\end{equation}

Our first main result is stated as follows.
\begin{theorem}[Intrinsic Riesz potential estimate]\label{thm:r1}
Let $u$ be a solution to \eqref{eq:u} with $p\in(p^*(n),2-\frac{1}{n+1}]$, where $p^*(n)$ is defined in \eqref{pstar}. Let $q\in(0,1)$ satisfy \eqref{rangeq}. Under the assumptions \eqref{ineq:growth}--\eqref{dini}, there exist constants $c\geq 1$ and $R_0\in(0,1/2]$, both depending only on $n$, $p$, $\nu$, $L$, $q$, and $\omega$, such that the following holds for a.e. $(x_0,t_0)\in \Omega_T$: If
\begin{equation}\label{eq:thm1}
 c \,\Big(\fint_{Q_{r_\lambda}^\lambda(x_0,t_0)} (|Du|+s)^q\,dxdt\Big)^{1/q} +c\, \int_{0}^{2r_\lambda} \frac{|\mu|(Q_\rho^\lambda(x_0,t_0))}{\rho^{n+1}} \frac{d\rho}{\rho}\leq \lambda,
\end{equation}
where $\lambda>0$ is a constant,  $r_\lambda:=\lambda^{(p-2)/2}r\in(0, R_0]$,  $Q_{2r_\lambda}^\lambda (x_0,t_0)\subset \Omega_T$, and $Q_\rho^\lambda(x_0,t_0)$ is the intrinsic cylinder defined in \eqref{eq:ic}, then it holds that
$$
|Du(x_0,t_0)|\leq \lambda.
$$
\end{theorem}
Theorem \ref{thm:r1} implies pointwise gradient estimates in standard parabolic cylinders as in Theorem \ref{thm:int} and Corollary \ref{cor:1} below.
\begin{theorem}[Pointwise gradient estimate via parabolic Riesz potential]\label{thm:int}
Let $u$ be a solution to \eqref{eq:u} with $p\in(p^*(n),2-\frac{1}{n+1}]$, where $p^*(n)$ is defined in \eqref{pstar}. Let $q\in(0,1)$ satisfy \eqref{rangeq}. Under the assumptions \eqref{ineq:growth}--\eqref{dini}, there exist constants $c\geq 1$ and $R_0\in(0,1/2]$, both depending only on $n$, $p$, $\nu$, $L$, $q$, and $\omega$, such that
\begin{equation}\label{eq:thm2}
    \begin{aligned}
    |Du(x_0,t_0)|&\leq c\, [\mathbf{I}_1^{|\mu|}(x_0,t_0,2r)]^{2/[(n+1)p-2n]}\\&\quad+c\, \Big(\fint_{Q_r(x_0,t_0)}(|Du|+s+1)^q\,dxdt\Big)^{2q/[2q-n(2-p)]}
\end{aligned}
\end{equation}
holds for a.e. $(x_0,t_0)\in\Omega_T$ and every  $Q_{2r}(x_0,t_0)\equiv B_{2r}(x_0)\times  (t_0-4r^2,t_0)\subset \Omega_T$ with $r\in(0, R_0]$, where $\mathbf{I}_1^{|\mu|}$ is the parabolic Riesz potential defined in \eqref{riesz2}.
\end{theorem}

\begin{corollary}[Pointwise gradient estimate via elliptic Riesz potential]\label{cor:1}
Let $u$ be a solution to \eqref{eq:u} with $p\in(p^*(n),2-\frac{1}{n+1}]$, where $p^*(n)$ is defined in \eqref{pstar} and assume that $\mu=\mu_0\otimes f$, where $\mu_0$ is a finite signed Radon measure on $\R^n$ and $f\in L^\infty(-T,0)$. Let $q\in(0,1)$ satisfy \eqref{rangeq}.  Under the assumptions \eqref{ineq:growth}--\eqref{dini}, there exist constants $c\geq 1$ and $R_0\in(0,1/2]$, both depending only on $n$, $p$, $\nu$, $L$, $q$, and $\omega$, such that
\begin{equation*}
    \begin{aligned}
    |Du(x_0,t_0)|&\leq c\,\|f\|_{L^\infty}^{1/(p-1)} [\mathbf{I}_1^{|\mu_0|}(x_0,2r)]^{1/(p-1)}\\
    &\quad+c\, \Big(\fint_{Q_r(x_0,t_0)}(|Du|+s+1)^q\,dxdt\Big)^{2q/[2q-n(2-p)]}
\end{aligned}
\end{equation*}
holds for a.e. $(x_0,t_0)\in\Omega_T$ and every  $Q_{2r}(x_0,t_0)\equiv B_{2r}(x_0)\times  (t_0-4r^2,t_0)\subset \Omega_T$ with $r\in(0, R_0]$, where $\mathbf{I}_1^{|\mu_0|}$ is the classical Riesz potential defined in \eqref{riesz1}.
\end{corollary}
\subsection{Gradient continuity results}

In \cite{kuusi2013mingione}, the authors proved a sufficient condition for gradient continuity in the case when $p\in(2-\frac{1}{n+1},2)$, namely, the Riesz potential $\mathbf{I}_1^{|\mu|}(x_0,t_0,r)\to 0$ uniformly with respect to $(x_0,t_0)$ when $r\to 0$. We extend that result to the case when $p\in(p^*(n),2-\frac{1}{n+1}]$.

\begin{theorem}[Gradient continuity via Riesz potential]\label{thm:cty}
Let $u$ be a solution to \eqref{eq:u} with $p\in(p^*(n),2-\frac{1}{n+1}]$, where $p^*(n)$ is defined in \eqref{pstar}. Assume that \eqref{ineq:growth}--\eqref{dini} are satisfied and that the functions
\begin{equation}\label{asp1}
(x,t)\;\mapsto \mathbf{I}_1^{|\mu|}(x,t,r) \text{ converge locally uniformly to zero in } \Omega_T \text{ as } r\;\to\;0.
\end{equation}
Then $Du$ is continuous in $\Omega_T$.
\end{theorem}

Recall the Lorentz space $L^{n+2,1}$ is the collection of measurable functions $f$ such that
$$
\int_0^\infty |\{(x,t):\,|f(x,t)|\ge h\}|^{\frac {1} {n+2}}\,{dh}<\infty.
$$
Theorem \ref{thm:cty} has the following corollary.

\begin{corollary}[Gradient continuity via Lorentz spaces]\label{thm1.3}
Let $u$ be a solution to \eqref{eq:u} with $p\in(p^*(n),2-\frac{1}{n+1}]$, where $p^*(n)$ is defined in \eqref{pstar}. Assume that
\eqref{ineq:growth}--\eqref{dini} are satisfied and that
\begin{equation}\label{asp2}
  \mu\in L^{n+2,1} \text{ holds locally in }\Omega_T.
\end{equation}
Then $Du$ is continuous in $\Omega_T$.
\end{corollary}

A further, actually immediate, corollary of Theorem \ref{thm:cty} concerns measures with certain density properties.
\begin{corollary}[Gradient continuity via density]\label{thm1.4}
Let $u$ be a solution to \eqref{eq:u} with $p\in(p^*(n),2-\frac{1}{n+1}]$, where $p^*(n)$ is defined in \eqref{pstar}. Assume that \eqref{ineq:growth}--\eqref{dini} are satisfied and that $\mu$ satisfies
\begin{equation}\label{asp3}
|\mu|(Q_\rho(x,t))\leq c_D\rho^{n+1}h(\rho)
\end{equation}
for every standard parabolic cylinder $Q_{\rho}(x,t)=B_\rho(x)\times (t_0-\rho^2, t_0)\subset\subset\Omega_T$, where $c_D$ is a positive constant and $h:[0,\infty)\to[0,\infty)$ is a function satisfying the Dini condition
\begin{equation}\label{asp4}
\int_0^R h(r)\,\frac{dr}{r}<\infty \text{ for some }R>0.
\end{equation}
Then $Du$ is continuous in $\Omega_T$.
\end{corollary}

We also establish the following measure density criterion to ensure gradient H\"older continuity, which is a parabolic generalization of Lieberman's result in \cite{MR1233190}. Recall that in parabolic setting, for any $\beta\in(0,1)$ and any set $\mathcal{C}\subset \R^{n+1}$, the H\"older space $C^{0,\beta}(\mathcal{C})$ is the collection of measurable functions $f$ such that
\begin{equation}\label{holdernorm}
\|f\|_{C^{0,\beta}(\mathcal{C})}:=\sup_{\mathcal{C}}|f|+\sup_{\substack{(x_1,t_1),(x_2,t_2)\in \mathcal{C} \\ (x_1,t_1)\neq (x_2,t_2)}}\frac{|f(x_1,t_1)-f(x_2,t_2)|}{|(x_1,t_1)-(x_2,t_2)|_{\text{par}}^\beta}<\infty,
\end{equation}
where
$$
|(x_1,t_1)-(x_2,t_2)|_{\text{par}}:=\max\{|x_1-x_2|,\sqrt{|t_1-t_2|}\}
$$
is the parabolic distance between those two points. Moreover, $C^{0,\beta}_{\text{loc}}(\mathcal{C})$ is defined as the collection of measurable functions $f$ such that $f\in C^{0,\beta}(\mathcal{K})$, for every compact set $\mathcal{K}\subset\subset \mathcal{C}$.
\begin{theorem}
[Gradient H\"older continuity via Riesz potential]\label{thm1.8}
Let $u$ be a solution to \eqref{eq:u} with $p\in(p^*(n),2-\frac{1}{n+1}]$, where $p^*(n)$ is defined in \eqref{pstar}. Assume that \eqref{ineq:growth}--\eqref{dini} are satisfied and that $\omega$, $\mu$ satisfies
\begin{equation}\label{asp5}
\omega(r)\leq c_D r^{\delta} \quad \text{and} \quad |\mu|(Q_\rho(x,t))\leq c_D\rho^{n+1+\delta}
\end{equation}
for every $r\in(0,1)$ and every standard parabolic cylinder $Q_{\rho}(x,t)=B_{\rho}(x)\times (t_0-\rho^2, t_0)\subset\subset\Omega_T$, where $c_D\geq 1$ and $\delta\in(0,1)$. Then there exists an exponent $\beta\in(0,1)$ depending only on $n$, $p$, $\nu$, $L$, $c_D$, and $\delta$, such that $Du\in C^{0,\beta}_{\text{loc}}(\Omega_T)$.
\end{theorem}
Let us give a brief description of the proofs.
We first prove decay estimates of the $L^q$-mean oscillation of the gradient for a solution to
the homogeneous equation with
$x$-independent nonlinearity
\begin{equation}\label{eq:v111}
v_t-\div (a(x_0,t, Dv))=0,
\end{equation}
where $q\in (0,1)$ and $x_0\in \R^n$ is fixed. Here by the $L^q$-mean oscillation of $Dv$ in a domain $\mathcal{C}\subset \R^{n+1}$, we mean
$$
\inf_{\Theta\in \R^n} \Big(\fint_{\mathcal{C}} |Dv-\Theta|^q\, dxdt\Big)^{1/q}.
$$
Our proof of the decay estimates adapts the singular iteration scheme in \cite[Section 3]{kuusi2013mingione} to the $L^q$ setting for $q\in(0,1)$. For the precise definition of the $L^q$-mean oscillation with $q\in(0,1)$ and some of its properties, see Section \ref{sec2.2}. We also refer the reader to \cite{luis2003estimates,choi2019gradient,dong2012gradient,dong2017c1,dong2020on,dong2021gradient,krylov2010on} for its applications in other problems.

Our proofs of the pointwise gradients estimates and the gradient continuity results are all based on the decay estimates for $Dv$ mentioned above and comparison estimates
between the original solution $u$ to \eqref{eq:u} and a solution $v$ to \eqref{eq:v111}.
As a bridge between $u$ and $v$, we introduce the solution $w$ to the homogeneous equation
$$
w_t-\div(a(x,t,Dw))=0
$$
in a cylinder $Q$ with the boundary condition $w=u$ on $\partial_{\text{par}} Q$. Under appropriate boundary condition on $v$, we obtain an $L^p$ bound for $Dw-Dv$, which originated from \cite[Lemma 4.3]{kuusi2012new}. We also utilize a comparison estimate between $u$ and $w$ in \cite[lemma 3.1]{park2020regularity}, which provides an $L^q$ bound for $Du-Dw$ in terms of $\mu$ for some $q\in(0,1)$. By proving a reverse H\"older type inequality for $Dw$, we establish an $L^q$ estimate for $Du-Dv$ for some $q\in(0,1)$.

With the decay estimates of the $L^q$-mean oscillation for $Dv$ and $L^q$ estimate for $Du-Dv$ in hand, we then borrow the idea in \cite{dong2017c1} by estimating the $L^{q}$-mean oscillation and adapt the exit time argument and iteration argument used, for instance, in \cite[Theorem 1.1]{kuusi2013mingione} to prove the pointwise gradient estimates. For gradient continuity results, we first prove a uniform decay estimate of the $L^q$-mean oscillation of $Du$ in Proposition \ref{lem:cty}. 
Then for the gradient H\"older continuity result, we show the decay rate of the $L^q$-mean oscillation of $Du$ and adapt Campanato's idea of characterizing H\"older continuity to the $L^q$ setting for some $q\in(0,1)$. Finally, for the gradient continuity result, we adapt the ``maximal iteration chain'' argument introduced in \cite[Theorem 1.5]{kuusi2013mingione} to our $L^q$ setting and apply the uniform decay of the $L^q$-mean oscillation of $Du$.

The rest of the paper is organized as follows. In the next section, we
collect basic notation and give the definition and some basic properties of the $L^q$-mean oscillation for $q\in(0,1)$. In Section \ref{sec3}, we
prove some decay estimates for the $L^{q}$-mean oscillation of the gradients of solutions to the homogeneous equation with $x$-independent nonlinearities. In Section \ref{sec4}, we derive some comparison estimates and give the proofs of Theorem \ref{thm:r1}--Corollary \ref{cor:1}. Finally, Section \ref{sec5} is devoted to the gradient continuity results
Theorem \ref{thm:cty}--Theorem \ref{thm1.8}.

\section{Notation and basic inequalities}
\subsection{Notation}\label{sec2.1}
In this paper, we adapt the same notation as in \cite{kuusi2013mingione} for comparison purposes. For completeness, we briefly record the notation that will be used throughout this paper.
For any vector $y=(y_1,\ldots,y_n)\in \R^n$, we define two different norms
$$
|y|:= \big(\sum_{i=1}^n y_i^2\big)^{1/2} \quad \text{and} \quad \|y\|:=\max_{1\leq i\leq n} |y_i|.
$$
These two norms are equivalent since
$$
\|y\| \leq |y| \leq \sqrt{n} \|y\|, \quad \forall \,y\in \R^n.
$$
We use $$B_r(x_0):=\{x\in\R^n: \,|x-x_0|<r\}$$ to denote the open Euclidean ball in $\R^n$ with center $x_0$ and radius $r$
and denote
$$
Q_r(x_0,t_0):=B_r(x_0)\times (t_0-r^2, t_0)
$$
as the standard parabolic cylinder with center $(x_0,t_0)$ and  radius $r$. For $\lambda>0$, we define the intrinsic cylinders as
$$ Q^\lambda_r(x_0,t_0):=B_r(x_0)\times (t_0-\lambda^{2-p} r^2, t_0).$$
Clearly, when $\lambda=1$, an intrinsic cylinder becomes a standard parabolic cylinder, namely, $Q^1_r(x_0,t_0)=Q_r(x_0,t_0)$.
For simplicity, we also denote
$$\delta Q_r^\lambda(x_0,t_0):=Q^\lambda_{\delta  r}(x_0,t_0)=B_{\delta r}(x_0)\times (t_0,\lambda^{2-p}\delta^2r^2,t_0).$$ Namely, the parameter $\delta>0$ before an intrinsic cylinder $Q_r^\lambda(x_0,r_0)$ should be viewed as a dilation factor of the radius $r$. We often denote $r_\lambda:=\lambda^{(p-2)/2} r$ and therefore
$$
Q_{r_\lambda}^\lambda(x_0,t_0)=Q^{\lambda}_{\lambda^{(p-2)/2} r}(x_0,t_0)=B_{\lambda^{(p-2)/2}r}(x_0)\times (t_0-r^2,t_0).
$$
A useful property is that when $p\in(1,2)$,
$$Q_{r_{\lambda_2}}^{\lambda_2}(x_0,t_0)\subset Q_{r_{\lambda_1}}^{\lambda_1}(x_0,t_0), \quad \text{if} \quad 0<\lambda_1\leq \lambda_2.$$
When there is no confusion and no need to specify the center, we also denote $Q_r^\lambda:=Q_r^\lambda(x_0,t_0)$.
For any cylindrical domain $\mathcal{C}=D\times (t_1,t_2)$ with $D\subset \R^n$, the parabolic boundary is defined as
$$\partial_{\text{par}} \mathcal{C}:=D\times \{t_1\} \cup \partial D\times [t_1,t_2).$$
The parabolic distance between two points is defined as $$|(x_1,t_1)-(x_2,t_2)|_{\text{par}}:=\max\{|x_1-x_2|,\sqrt{|t_1-t_2|}\}$$
and the corresponding parabolic distance between two sets is defined as
$$ \text{dist}_{\text{par}}(\mathcal{A}_1,\mathcal{A}_2)=\inf \{|(x_1,t_1)-(x_2,t_2)|_{\text{par}}:\, (x_1,t_1)\in \mathcal{A}_1,\, (x_2,t_2)\in \mathcal{A}_2\}.$$
Next, for any measurable mapping $g:\mathcal{A} \subset \R^{n+1} \to \R^n$, we denote its integral average as
$$
\fint_\mathcal{A} g \,dxdt:=\frac{1}{|\mathcal{A}|} \int_\mathcal{A} g(x,t) \,dxdt
$$
and we denote its oscillation as
$$
\osc\limits_\mathcal{A} \,g:= \sup_{(x,t),\,(x_0,t_0)\in \mathcal{A}} |g(x,t)-g(x_0,t_0)|.
$$
Finally, throughout the paper, we denote by $c, \,c',\,c''$ some general constants which may differ from line to line. We also use $c_1,\,c_2, \, c_3,\,\ldots$ to denote specific constants which may be used later.

\subsection{Definition of the \texorpdfstring{$L^q$}{Lq}-mean oscillation and some basic inequalities}\label{sec2.2}
For $q\in(0,1)$, we first recall some basic inequalities in one dimension:
\begin{equation}\label{eq:1d}
    (a+b)^q\leq a^q+b^q,\quad (a+b)^{1/q}\leq 2^{1/q-1} (a^{1/q} +b^{1/q}) \quad \forall \,a,b>0.
\end{equation}
Therefore, for any $\Theta_1, \,\Theta_2\in \R^k$, where $k$ is a positive integer, we have
\begin{equation}\label{eq:kd}
    |\Theta_1+\Theta_2|^q\leq (|\Theta_1|+|\Theta_2|)^q\leq |\Theta_1|^q+|\Theta_2|^q.
\end{equation}
We now consider a measurable function $F:\,\R^{n+1} \to \R^k$ for some integer $k\geq 1$. For simplicity, we assume that $F\in L^1_{\text{loc}}(\R^{n+1}:\R^k)$. For any $q\in (0,1)$ and any bounded domain $\mathcal{C}\subset \R^{n+1}$, we define the $L^q$-mean oscillation of $F$ on $\mathcal{C}$ as
$$
\phi_q(F,\mathcal{C}):= \inf_{\Theta\in \R^k} \Big(\fint_{\mathcal{C}} |F(x,t)-\Theta|^q \,dxdt\Big)^{1/q}.
$$
By \eqref{eq:kd} and the fact that $F\in L^1(\mathcal{C})$, we know that the function
$$ h(\Theta):=\fint_{\mathcal{C}} |F(x,t)-\Theta|^q \,dxdt$$
is a continuous function of $\Theta\in \R^k$ satisfying
$\lim_{|\Theta|\to \infty} h(\Theta)=\infty.$ Therefore the minimum of $h(\cdot)$ can be attained in $\R^k$ and we choose $\mathbf{m}(F,\mathcal{C})\in \R^k$ such that
$$
\Big(\fint_{\mathcal{C}} |F(x,t)-\mathbf{m}(F,\mathcal{C})|^q \,dxdt\Big)^{1/q}=\phi_q(F,\mathcal{C}).
$$
Next, we prove some useful properties related to the $L^q$-mean oscillation.
By \eqref{eq:kd}, we have
$$ |\mathbf{m}(F,\mathcal{C})|^q\leq |F(x,t)-\mathbf{m}(F,\mathcal{C})|^q+|F(x,t)|^q.$$
By taking the average over $(x,t)\in \mathcal{C}$, taking the $q$-th root, and using \eqref{eq:1d}, we obtain
\begin{equation}\label{eq:mbound}
|\mathbf{m}(F,\mathcal{C})|\leq 2^{1/q-1}\phi_q(F,\mathcal{C})+2^{1/q-1}\Big(\fint_{\mathcal{C}}F^q\,dxdt\Big)^{1/q}\leq 2^{1/q} \Big(\fint_{\mathcal{C}}F^q\,dxdt\Big)^{1/q}.
\end{equation}
For two bounded domains $\mathcal{C}_1\subset\mathcal{C}_2\subset \R^{n+1}$,
using the same argument as above, we also have
\begin{equation}\label{eq:mdiff}
\begin{aligned}
|\mathbf{m}(F,\mathcal{C}_1)-\mathbf{m}(F,\mathcal{C}_2)|&\leq 2^{1/q-1} \phi_q(F,\mathcal{C}_1)+2^{1/q-1} \Big(\fint_{\mathcal{C}_1} |F-\mathbf{m}(F,\mathcal{C}_2)|^q \,dxdt\Big)^{1/q}\\
&\leq 2^{1/q-1} \phi_q(F,\mathcal{C}_1)+2^{1/q-1} \Big(\frac{|\mathcal{C}_2|}{|\mathcal{C}_1|}\Big)^{1/q}
\phi_q(F,\mathcal{C}_2).
\end{aligned}
\end{equation}

\section{Gradient estimates for homogeneous equations}\label{sec3}

In this section, we derive decay estimates of the $L^q$-mean oscillation of gradients of solutions to the homogeneous equations of the type
\begin{equation}
                        \label{eq1.01}
v_t-\div (a_0(t,Dv))=0
\end{equation}
in a given cylinder $Q=B\times (t_1,t_2)$, where $a_0=a_0(\tau, \xi)$ is a vector field independent of $x$ satisfying conditions \eqref{ineq:growth} and \eqref{ineq:elliptic} for some $s\geq0$, $L\geq \nu>0$, and $p\in(1,2]$.
First, we recall an oscillation estimate given in \cite[Theorem 3.2]{kuusi2013mingione}.
\begin{theorem}\label{thm:osc}
Suppose that $v$ is a solution to \eqref{eq1.01} in a given cylinder $Q$ under assumptions \eqref{ineq:growth} and \eqref{ineq:elliptic}. If $p\in(1,2]$, $\lambda>0$, and
$$s+\sup_{Q_r^\lambda}\|Dv\|\leq A\lambda$$
holds for a constant $A\geq 1$ and an intrinsic cylinder $Q_r^\lambda\subset Q$, then there exists a constant $\alpha\in(0,1)$ depending only on $n$, $p$, $\nu$, $L$, $A$, such that
$$|Dv(x_1,t_1)-Dv(x_2,t_2)|\leq 4\sqrt{n} A \lambda \Big(\frac{\rho}{r}\Big)^\alpha$$
holds for any $(x_1,t_1),\, (x_2,t_2)\in Q_\rho^\lambda$, where $Q_\rho^\lambda \subset Q_r^\lambda$ is another intrinsic cylinder with the same center.
Moreover, if $p\in(\frac{2n}{n+2},2]$, then $Dv$ is H\"older continuous in $Q$.
 \end{theorem}

The main goal of this section is to derive the following decay estimate of the $L^q$-mean oscillation of the gradient with $q\in(0,1)$.

\begin{theorem}\label{thm:iter1}
Let $p\in(1,2]$. Suppose that $v$ is a solution to \eqref{eq1.01} in an intrinsic cylinder $Q_r^\lambda$ under assumptions \eqref{ineq:growth} and \eqref{ineq:elliptic} and that $A$, $B$, $q$, and $\gamma$ are constants satisfying
$A,\,B\geq 1$ and $q,\,\gamma \in(0,1)$. Then there exist constants $\delta_{\gamma}\in(0,1/2)$  depending on $n$, $p$, $\nu$, $L$, $A$, $B$, $\gamma$, $q$, and $\xi\in(0,1/4)$ depending only on $n$, $p$, $\nu$, $L$, $A$, $B$, $\gamma$, such that if
\begin{equation}\label{lower1}
\lambda\leq \max\Big\{\frac{s}{\xi }, B\sup_{\delta_\gamma Q_{r}^\lambda}\|Dv\|\Big\}, \quad s+\sup_{Q_r^\lambda}\|Dv\|\leq A\lambda,
\end{equation}
where $\lambda>0$ is a constant, then
\begin{equation}\label{eq3.102}
    \phi_q(Dv,\delta_\gamma Q_r^\lambda)\leq \gamma \phi_q(Dv,Q_r^\lambda).
\end{equation}
Moreover, there exist constants $\alpha\in(0,1)$ depending only on $n$, $p$, $\nu$, $L$, $A$, $B$, $\gamma$, but not on $q$, and $c(A,B)\geq 1$ depending on $n$, $p$, $\nu$, $L$, $A$, $B$, $\gamma$, $q$, such that
$$\delta_\gamma=\frac{\gamma^{1/\alpha}}{c(A,B)}.$$
Also, the estimate \eqref{eq3.102} still holds when replacing $\delta_\gamma$ with a smaller number.
\end{theorem}
In the proof of gradient continuity results in Section \ref{sec5}, we need a different version of Theorem \ref{thm:iter1} under a stronger condition as follows, which gives a more precise dependence of $\delta_\gamma$ in terms of $A$, $B$, and $\gamma$, namely,
$$\delta_\gamma= \frac{\gamma^{1/\alpha_2}}{c(A)B^{1/\alpha_1}}$$
for some exponent $\alpha_1\in(0,1)$ independent of $B$ and some $\alpha_2\in(0,1)$ independent of $B$ and $q$.
\begin{theorem}\label{thm:iter2}
Let $p\in(1,2]$. Suppose that $v$ is a solution to \eqref{eq1.01} in an intrinsic cylinder $Q_r^\lambda$ under assumptions \eqref{ineq:growth} and \eqref{ineq:elliptic} and that $A$, $B$, $q$, and $\gamma$ are constants satisfying
$A,\,B\geq 1$ and $q,\,\gamma \in(0,1)$. Then there exists a constant $\delta_{\gamma}\in(0,1/2)$  depending only on $n$, $p$, $\nu$, $L$, $A$, $B$, $\gamma$, $q$, such that if
\begin{equation*}
   \lambda\leq  B\sup_{\delta_\gamma Q_{r}^\lambda}\|Dv\|, \quad s+\sup_{Q_r^\lambda}\|Dv\|\leq A\lambda,
\end{equation*}
where $\lambda>0$ is a constant, then
\begin{equation}\label{eq3.202}
    \phi_q(Dv,\delta_\gamma Q_r^\lambda)\leq \gamma \phi_q(Dv,Q_r^\lambda).
\end{equation}
Moreover, there exist constants $\alpha_1\in(0,1)$ and $c(A)\geq 1$, both depending only on $n$, $p$, $\nu$, $L$, $A$, $\gamma$, $q$, but not on $B$, and $\alpha_2\in(0,1)$ depending only on $n$, $p$, $\nu$, $L$, $A$, $\gamma$, but not on $q$ and $B$, such that
\begin{equation}\label{eq3.203}
    \delta_\gamma=  \frac{\gamma^{1/\alpha_2}}{c(A)B^{1/\alpha_1}}.
\end{equation}
Also, the estimate \eqref{eq3.202} still holds when replacing $\delta_\gamma$ by a smaller number.
\end{theorem}
As in \cite{kuusi2013mingione}, we first assume $s>0$. The next two De Giorgi-Nash-Moser type estimates for $Dv$ can be found in \cite{kuusi2013mingione} for vector field $a_0(\cdot)$ with no time dependence. However, they still hold here since their proofs only rely on differentiating the equation with respect to the space variable. See \cite[Remark 3.6]{kuusi2013mingione}.

\begin{proposition}\label{height}
Assume $s>0$ and $\lambda>0$. Suppose that
\begin{equation}\label{1sup}
    s+\sup_{Q_r^\lambda}\|Dv\|\leq A\lambda
\end{equation}
holds for some constant $A\geq 1$. There exists a constant $\sigma\in(0,1/2)$ depending only on $n$, $p$, $\nu$, $L$, $A$ such that if either
\begin{equation}\label{meas1}
    |Q_r^\lambda\cap\{D_{x_i}v<\lambda/2\}|\leq \sigma |Q_r^\lambda|
\end{equation}
or
\begin{equation}\label{meas2}
    |Q_r^\lambda\cap\{D_{x_i}v>-\lambda/2\}|\leq \sigma |Q_r^\lambda|
\end{equation}
holds for some $i\in\{1,\ldots,n\}$,
then
$$
|D_{x_i}v|\geq \lambda/4 \quad \text{a.e. in }  Q_{r/2}^\lambda.
$$
\end{proposition}

\begin{proposition}\label{prop:deg}
Assume $s>0$ and $\lambda>0$. Suppose that \eqref{1sup} holds and that neither \eqref{meas1} nor \eqref{meas2} is satisfied for the constant $\sigma$ in Proposition \ref{height}. Then there exists a constant $\eta\in (1/2,1)$ depending only on $n$, $p$, $\nu$, $L$, $A$, such that

\begin{equation}\label{eq:deg}
    \|Dv\|\leq \eta A\lambda \quad \text{a.e. in } Q_{\sigma r/2}^\lambda.
\end{equation}
\end{proposition}
Next, we prove a decay result for the $L^q$-mean oscillation of gradients of solutions to linear parabolic equations with $q\in(0,1)$.
\begin{lemma}\label{lem:linear}
Suppose that $\tilde{u}\in L^2(-1,0;\,W^{1,2}(B_1(0)))$ is a solution to the following linear parabolic equation
\begin{equation}\label{linear}
 \tilde{u}_t-\textup{div}(A(x,t)D\tilde{u})=0,
 \end{equation}
where the matrix $A(x,t)$ has measurable entries and satisfies
$$
\nu_0|\xi|^2\leq \langle A(x,t)\xi,\xi \rangle, \quad |A(x,t)|\leq L_0
$$
for any $\xi \in \mathbb{R}^n$, where $0<\nu_0\leq L_0$ are fixed constants. Let $q\in(0,1)$. Then there exist constants $c_1, c_2\geq 1$ depending on $n$, $\nu_0$, $L_0$, $q$, and  $\beta_0\in(0,1)$ depending only on $n$, $\nu_0$, $L_0$, such that

\begin{equation}\label{eq3.111}
    \sup_{Q_{1/2}} |\tilde{u}|\leq c_1 \Big(\fint_{Q_1}|\tilde{u}|^q \,dxdt\Big)^{1/q}
\end{equation}
 and
 \begin{equation}\label{eq3.112}
    \phi_q(\tilde{u},Q_\delta)\leq c_2\delta^{\beta_0} \phi_q(\tilde{u},Q_1)
 \end{equation}
hold for any $\delta\in(0,1)$. Here for each $\rho>0$, $Q_\rho\equiv B_{\rho}(0)\times (-\rho^2,0)$ stands for the standard parabolic cylinder centered at $0$ with radius $\rho$.
\end{lemma}
\begin{proof}
First, the estimate \eqref{eq3.111} follows  from standard local boundedness estimates for nondegenerate parabolic equations (see for example (3.26) in \cite[Lemma 3.1]{kuusi2014the}) and a standard interpolation and iteration argument. See also \cite[Theorems 6.17]{lieberman1996second}.
Next we give the proof of \eqref{eq3.112} for $\delta\in(0,1/4)$.
From the classical De Giorgi-Nash-Moser theory for linear parabolic equations (see for example \cite[Theorems 6.28-6.29]{lieberman1996second}), there exist constants $c_0\geq 1$ and $\beta_0\in(0,1)$, both depending on $n$, $\nu_0$, $L_0$, such that $\tilde{u}\in C^{\beta_0}(Q_{1/4})$ and that
$$
\|\tilde{u}\|_{C^{\beta_0}(Q_{1/4})}\leq c_0 \sup_{Q_{1/2}}|\tilde{u}|,
$$
where $\|\cdot\|_{C^{\beta_0}}$ is the parabolic H\"older norm defined in \eqref{holdernorm}.
Thus, using the last estimate and the assumption that $\delta\in(0,1/4)$, we have
\begin{equation}\label{eq3.113}
\begin{aligned}
  &\phi_q(\tilde{u},Q_\delta)\leq
  \Big(\fint_{Q_\delta} |\tilde{u}-\tilde{u}(0,0)|^q\,dxdt\Big)^{1/q} \leq  (2\delta)^{\beta_0} \|\tilde{u}\|_{C^{\beta_0}(Q_{1/4})} \leq c\delta^{\beta_0}\sup_{Q_{1/2}} |\tilde{u}|,
\end{aligned}
\end{equation}
for some constant $c=c(n, \nu_0, L_0)>0$. Combining \eqref{eq3.113} and \eqref{eq3.111}, we get
$$
\phi_q(\tilde{u},Q_\delta)\leq c_2 \delta^{\beta_0}\Big(\fint_{Q_1}|\tilde{u}|^q\, dxdt\Big)^{1/q},
$$
and therefore \eqref{eq3.112} follows by replacing $\tilde{u}$ with $\tilde{u}-\mathbf{m}(\tilde{u},Q_1)$ since  $\tilde{u}-\mathbf{m}(\tilde{u},Q_{1})$ is still a solution to \eqref{linear} and $\phi_q(\tilde{u},Q_\delta)\equiv \phi_q(\tilde{u}-\mathbf{m}(\tilde{u},Q_{1}),Q_\delta)$.

Finally, for $\delta\in[1/4,1)$, the proof of \eqref{eq3.112} follows standard manipulations as follows. Recalling the definition of $\mathbf{m}$, we have
\begin{align*}
    &\phi_q(\tilde{u},Q_\delta)\leq  \Big(\fint_{Q_\delta} |\tilde{u}-\mathbf{m}(\tilde{u},Q_1)|^q\,dxdt\Big)^{1/q}
     \\&\leq
     \Big(\frac{|Q_1|}{|Q_\delta|}\Big)^{1/q} \Big(\fint_{Q_1} |\tilde{u}-\mathbf{m}(\tilde{u},Q_1)|^q\,dxdt\Big)^{1/q}
     \leq 4^{(n+2)/q} \phi_q(\tilde{u},Q_1).
\end{align*}
The proof is now completed.
\end{proof}
Combining Proposition \ref{height} and Lemma \ref{lem:linear}, we obtain the following result by using a scaling argument.
\begin{proposition}\label{prop:iter1}
Suppose that $s>0$, $\lambda>0$, and that \eqref{1sup} holds. There exist constants $\beta\in(0,1)$ depending on $n$, $p$, $\nu$, $L$, $A$, and $c_d\geq 1$ depending on $n$, $p$, $\nu$, $L$, $A$, $q$,  such that if either \eqref{meas1} or \eqref{meas2} holds for  the constant $\sigma$ in Proposition \ref{height} and some $i\in\{1,\ldots,n\}$, then
\begin{equation}\label{eq3.131}
    \phi_q(Dv,Q_{\delta r}^\lambda) \leq c_d \delta^{\beta} \phi_q(Dv,Q_{r}^\lambda).
\end{equation}

\end{proposition}
\begin{proof}
We will only give the proof for $\delta\in(0,1/2)$ since the proof for $\delta\in[1/2,1)$ follows by standard manipulations as in Lemma \ref{lem:linear}.
From Proposition \ref{height}, we have
\begin{equation}\label{bounds1}
    \lambda/4\leq \|Dv(x,t)\|\leq s+\|Dv(x,t)\|\leq A\lambda \quad \forall \,(x,t)\in Q_{r/2}^\lambda.
\end{equation}
Then we rescale the solution in the cylinder $Q_1$, namely,
\begin{equation}\label{rescale}
\tilde{v}(x,t):=\frac{1}{r_1}v(r_1 x,\lambda^{2-p}r_1^2t), \quad (x,t)\in Q_1,
\end{equation}
where $r_1=r/2$.
Thus $\tilde{v}$ satisfies
\begin{equation}\label{eq:tildev}
   \lambda^{p-2} \tilde{v}_t-\text{div}\, \tilde{a}_0(t,D\tilde{v})=0,
\end{equation}
where
$$\tilde{a}_0(t,z)=a_0(\lambda^{2-p}r_1^2t,z), \quad t\in(-1,0),\; z\in \mathbb{R}^n.$$
Also, \eqref{bounds1} implies
\begin{equation}\label{bounds2}
    \lambda/4\leq \|D\tilde{v}(x,t)\|\leq s+\|D\tilde{v}(x,t)\|\leq A\lambda \quad \forall \,(x,t)\in Q_1.
\end{equation}
Since $D\Tilde{v}$ is bounded, we know that
$$ D\Tilde{v}\in L^2_{\text{loc}}(-1,0;\,W^{1,2}_{\text{loc}}(B_1,\mathbb{R}^n))\cap C^0(-1,0;\,L^2_{\text{loc}}(B_1,\mathbb{R}^n)).$$
See \cite[Chapter 8, Section 3]{dibenedetto1993degenerate} for details. Therefore we can differentiate \eqref{eq:tildev} in $x_i$-direction for each $i\in\{1,\ldots,n\}$ and get
\begin{equation}\label{eq:dv_i}
(\tilde{v}_{x_i})_t-\text{div}( A(x,t)D\tilde{v}_{x_i})=0, \quad \text{where } A(x,t):=\lambda^{2-p} \partial_z \tilde{a}_0(t,D\tilde{v}(x,t)).
\end{equation}
By \eqref{bounds2}, the matrix $A(x,t)$ is uniformly elliptic, namely, for any $\eta\in\R^n$
\begin{equation*}
    \langle A(x,t)\eta,\eta\rangle\geq \nu\lambda^{2-p}(|D\tilde{v}(x,t)|^2+s^2)^\frac{p-2}{2} |\eta|^2\geq \nu (\sqrt{n}A)^{p-2} |\eta|^2,
\end{equation*}
and
$$
|A(x,t)|\leq L \lambda^{2-p} (|D\tilde{v}(x,t)|^2+s^2)^\frac{p-2}{2} \leq  4^{2-p}L.
$$
holds whenever $\xi\in \mathbb{R}^n$ for some constant $\Lambda\geq 1$ depending only on $n$, $p$, $\nu$, $L$, $A$.
Therefore we can apply Lemma \ref{lem:linear} to $\tilde{v}_{x_i}$  and get
\begin{equation}\label{eq3.132}
    \phi_q(\tilde{v}_{x_i},Q_{\delta_0})\leq c\delta_0^\beta \phi_q(\tilde{v}_{x_i},Q_1)
 \end{equation}
for any $\delta_0\in(0,1)$ and $i\in\{1,\ldots,n\}$, where $\beta=\beta(n,p,\nu,L,A)\in(0,1)$ and $c=c(n,p,\nu,L,A,q)\geq 1$.
Let $\Theta_0:=(\mathbf{m}(\tilde{v}_{x_1},Q_{\delta_0}),\ldots,\mathbf{m}(\tilde{v}_{x_n},Q_{\delta_0}))$
and $\Theta_1:=(\mathbf{m}(\tilde{v}_{x_1},Q_{1}),\ldots,\mathbf{m}(\tilde{v}_{x_n},Q_{1}))$.
Using the triangle inequality and Jensen's inequality, we have
\begin{equation}\label{eq:deltaq}
\phi_q(D\tilde{v},Q_{\delta_0})\leq \Big(\fint_{Q_{\delta_0}} |D\tilde{v}-\Theta_0|^q\,dxdt\Big)^{1/q}
\leq n^{1/q-1}\sum_{i=1}^n \phi_q(\tilde{v}_{x_i},Q_{\delta_0}).
\end{equation}
On the other hand, by the definition of $\phi_q$, we have
$$\phi_q(\tilde{v}_{x_i},Q_1)\leq\phi_q(D\tilde{v},Q_{1}),\quad \forall \,i\in\{1,\ldots,n\}$$
and therefore
\begin{equation}\label{eq:1q}
    \sum_{i=1}^n \phi_q(\tilde{v}_{x_i},Q_{1})\leq n \phi_q(D\tilde{v},Q_{1}).
\end{equation}
By \eqref{eq3.132}, \eqref{eq:deltaq}, and \eqref{eq:1q}, we obtain
\begin{equation}\label{eq3.133}
     \phi_q(D\tilde{v},Q_{\delta_0})\leq c'\delta_0^\beta \phi_q(D\tilde{v},Q_{1})
\end{equation}
for some $c'=c'(n,p,\nu,L,A,q)\geq 1$.
 Rescaling back in $v$, \eqref{eq3.133} becomes
 \begin{equation}\label{eq3.134}
     \phi_q(D{v},Q_{\delta r}^\lambda)\leq c'\delta^\beta \phi_q(D{v},Q_{r/2}^\lambda),
\end{equation}
where $\delta=\delta_0/2\in(0,1/2).$
Moreover, by the definition of $\phi_q$, we see that
\begin{equation}\label{eq3.135}
  \phi_q(D{v},Q_{r/2}^\lambda) \leq  \Big(\fint_{Q_{r/2}^\lambda} |D{v}-\mathbf{m}(D{v},Q_{r}^\lambda)|^q\,dxdt\Big)^{1/q} \leq 2^{(n+2)/q}\phi_q(D{v},Q_{r}^\lambda).
\end{equation}
Combining \eqref{eq3.134} and \eqref{eq3.135}, we obtain \eqref{eq3.131} for any $\delta\in(0,1/2)$. The proof is now completed.
\end{proof}

Similarly, we have the following result for relatively large $s$.
\begin{proposition}\label{prop:iter2}
Let $\lambda,A>0$ be constants.
Assume that
\begin{equation}\label{eq3.121}
    \sup_{Q_r^\lambda} \|Dv\|\leq A\lambda \quad \text{and} \quad 0<\xi \lambda\leq s\leq \xi_1 A\lambda, \quad \text{where }0<\xi\leq \xi_1 A.
\end{equation}
Then there exist constants $\beta_1\in(0,1)$ depending on $n$, $p$, $\nu$, $L$, $A$, $\xi$, $\xi_1$, and $\tilde{c}_d\geq 1$ depending on $n$, $p$, $\nu$, $L$, $A$, $\xi$, $\xi_1$, $q$, such that
\begin{equation}\label{eq3.122}
    \phi_q(Dv,Q_{\delta r}^\lambda)\leq \tilde{c}_d \delta^{\beta_1} \phi_q(Dv,Q_r^\lambda)
\end{equation}
holds for any $\delta\in(0,1)$.
\end{proposition}
\begin{proof}
First, we rescale the solution in $Q_1$ as in \eqref{rescale}, but this time with $r_1=r$. Then the rescaled solution $\tilde{v}$ stills satisfies \eqref{eq:tildev} in $Q_1$ and $\tilde{v}_{x_i}$ solves \eqref{eq:dv_i}. This time we can use \eqref{eq3.121} to get
\begin{equation*}
    \langle A(x,t)\eta,\eta\rangle\geq \nu\lambda^{2-p}(|D\tilde{v}(x,t)|^2+s^2)^\frac{p-2}{2} |\eta|^2\geq \nu (n+\xi_1^2)^\frac{p-2}{2}A^{p-2} |\eta|^2,
\end{equation*}
and
$$
|A(x,t)|\leq L \lambda^{2-p} (|D\tilde{v}(x,t)|^2+s^2)^\frac{p-2}{2} \leq  \xi^{p-2}L.
$$
 Then, we can proceed exactly as in Proposition \ref{prop:iter1} and eventually obtain \eqref{eq3.122}.
\end{proof}

Now we are ready to give the proofs of
Theorems \ref{thm:iter1} and \ref{thm:iter2}.
\begin{proof}[Proof of Theorem \ref{thm:iter1}]
Our proof closely follows that of \cite[Theorem 3.1]{kuusi2013mingione} using the ``singular iteration'' argument with
Propositions \ref{prop:iter1}, \ref{prop:iter2}, and \ref{prop:deg} in place of
\cite[Propositions 3.8, 3.9, and 3.11]{kuusi2013mingione}.  More specifically, the case when Proposition \ref{prop:iter1} occurs is called the nonsingular alternative and the case when Proposition \ref{prop:deg} occurs is called the singular alternative. The main idea is to construct a chain of intrinsic cylinders where the singular alternative occurs and show that the chain will stop in a finite time because of the occurrence of the nonsingular alternative. As in \cite{kuusi2013mingione}, we first assume $s>0$ in Steps 1-5 below.

\emph{Step 1: Stopping time for the singular iteration and the choice of $\xi$.}
Let $\eta_1:=(1+\eta)/2\in(0,1)$, where $\eta$ is given in Proposition \ref{prop:deg}, and  therefore $\eta_1$ depends only on $n$, $p$, $\nu$, $L$, and $A$. Define $m$ as the smallest integer such that
\begin{equation}\label{def:stopping}
\eta_1^m A<\frac{1}{2B}.
\end{equation}
Thus we know that $m$ depends on
$n$, $p$, $\nu$, $L$, $A$, $B$, and we have
\begin{equation}\label{ineq:stopping}
\frac{\log(2AB)}{-\log(\eta_1)}<m\leq 1+\frac{\log(2AB)}{-\log(\eta_1)}.
\end{equation}
Next we choose
\begin{equation}\label{eq:choicexi}
\xi:=\min\{1/8,\,(\eta_1-\eta)\eta_1^{2m} A\},
\end{equation}
which also depends on $n$, $p$, $\nu$, $L$, $A$, and $B$.

\emph{Step 2: The first nonsingular case: $\xi\lambda\leq s$.} Proposition \ref{prop:iter2} implies that \eqref{eq3.102} holds whenever
\begin{equation}\label{eq:delta1}
    \delta_\gamma\leq (\gamma/\tilde{c}_d)^{1/\beta_1}.
\end{equation} Note that
here $\beta_1 \in (0,1)$ depends only on $n$, $p$, $\nu$, $L$, $A$, $B$, and $\tilde{c}_d\geq 1$ depends on $n$, $p$, $\nu$, $L$, $A$, $B$, and $q$. From now on, we assume instead
\begin{equation}\label{eq:3.501}
 s< \xi\lambda\leq (\eta_1-\eta)\eta_1^{2m} A\lambda,
\end{equation}
where the second inequality always holds by the definition of $\xi$ in \eqref{eq:choicexi}.

\emph{Step 3: The singular iteration.}
Given a cylinder $Q_r^\lambda$, where \eqref{1sup} holds, then by Propositions \ref{prop:iter1} and \ref{prop:deg}, we know that at least one of \eqref{eq3.131} and  \eqref{eq:deg} must hold. The case when Proposition \ref{prop:iter1} applies is called the nonsingular alternative, while the case when
Proposition \ref{prop:deg} applies is called the singular alternative. We now let $\sigma_1=\sigma/2$, where $\sigma$ is the constant in Proposition \ref{height}, and therefore $\sigma_1$ depends only on $n$, $p$, $\nu$, $L$, and $A$. We define the sequences
$\lambda_0:=\lambda$, $\lambda_{i+1}:=\eta_1 \lambda_i$ and  $R_0:=r$, $R_{i+1}:=\sigma_1 R_i$.
Exactly as in the proof of \cite[Theorem 3.1]{kuusi2013mingione}, we build the singular iteration scheme by induction. Assume that the singular alternative holds in $Q_{R_{i-1}}^{\lambda_{i-1}}$ whenever $i\in\{1,\ldots, j\}$ for some $1\leq j\leq 2m$. Arguing exactly as in \cite{kuusi2013mingione}, under the condition \eqref{eq:3.501} we obtain
\begin{equation}\label{eq:singiter}
    s+\sup_{Q_{R_j}^{\lambda_j}} \|Dv\|\leq A \lambda_j,
\end{equation}
which verifies the upper bound \eqref{1sup} in the cylinder $Q_{R_j}^{\lambda_j}$. Then we can determine whether the singular alternative or the nonsingular alternative occurs in $Q_{R_j}^{\lambda_j}$.

\emph{Step 4: The second nonsingular case.}
Let
\begin{equation}\label{def:del}
\delta_\gamma\leq \underline{\delta}:=(\eta_1^{(2-p)/2} \sigma_1)^{m+1},
\end{equation}
and define
\begin{equation}\label{def:tm}
  \tilde{m}=\min\{k\in \mathbb{N}:\; \text{the singular alternative does not hold in } Q_{R_k}^{\lambda_k}\}.
\end{equation}
By \eqref{lower1}, \eqref{eq:3.501}, and the definition of $m$ in \eqref{def:stopping}, arguing exactly as in \cite{kuusi2013mingione}, we have $\tilde{m}\leq m$. Moreover, by \eqref{eq:singiter} and the second inequality in \eqref{lower1}, it holds that
\begin{equation}\label{eq:singiter2}
   s+\sup_{Q_{R_j}^{\lambda_j}} \|Dv\|\leq A \lambda_j,\quad \text{for } j\in\{0,\ldots,\tilde{m}\}.
\end{equation}
Since \eqref{eq:singiter2} in particular holds for $j=\tilde{m}$, we can apply Proposition \ref{prop:iter1} in $Q_{R_{\tilde{m}}}^{\lambda_{\tilde{m}}}$ and get
\begin{equation}\label{eq3.141}
      \phi_q(Dv, Q_{\delta R_{\tilde{m}}}^{\lambda_{\tilde{m}}})\leq c_d \delta^{\beta} \phi_q(Dv,Q_{R_{\tilde{m}}}^{\lambda_{\tilde{m}}})
\end{equation}
for any $\delta\in(0,1)$, where $c_d\geq 1$ is a constant depending on $n$, $p$, $\nu$, $L$, $A$, and $q$.
Moreover, using the fact that $\tilde{m}\leq m$, we have
\begin{equation}\label{eq3.142}
    \begin{aligned}
         \phi_q(Dv,Q_{R_{\tilde{m}}}^{\lambda_{\tilde{m}}})&\leq \Big(\frac{|Q_r^\lambda|}{|Q_{R_{\tilde{m}}}^{\lambda_{\tilde{m}}}|}\Big)^{1/q} \Big(\fint_{Q_r^\lambda} |Dv-\mathbf{m}(Dv,Q_r^\lambda)|^q\,dxdt\Big)^{1/q}
         \\&\leq \sigma_1^{-m(n+2)/q}\eta_1^{-m(2-p)/q} \phi_q(Dv,Q_r^\lambda).
    \end{aligned}
\end{equation}
Furthermore, by the definition of $\underline{\delta}$ in \eqref{def:del} and again the fact that $\tilde{m}\le m$, we have $Q_{\delta\underline{\delta} r}^\lambda\subset Q_{\delta R_m}^{\lambda_m}\subset
Q_{\delta R_{\tilde{m}}}^{\lambda_{\tilde{m}}}
$
and thus
\begin{equation}\label{eq3.143}
\begin{aligned}
     \phi_q(Dv, Q_{ \delta\underline{\delta} r}^{\lambda})&\leq
      \Big( \frac{|Q_{\delta R_{\tilde{m}}}^{\lambda_{\tilde{m}}}|}{|Q_{ \delta\underline{\delta} r}^{\lambda}|}  \Big)^{1/q}
      \Big(\fint_{Q_{\delta R_{\tilde{m}}}^{\lambda_{\tilde{m}}}} |Dv-\mathbf{m}(Dv,Q_{\delta R_{\tilde{m}}}^{\lambda_{\tilde{m}}})|^q\,dxdt\Big)^{1/q}
      \\& \leq \underline{\delta}^{-(n+2)/q}  \phi_q(Dv,Q_{\delta R_{\tilde{m}}}^{\lambda_{\tilde{m}}}).
\end{aligned}
\end{equation}
Combining \eqref{eq3.141}, \eqref{eq3.142}, and \eqref{eq3.143}, we obtain
\begin{align*}
          \phi_q(Dv, Q_{ \delta\underline{\delta} r}^{\lambda})\leq c_d \delta^\beta \Big[\underline{\delta}^{-(n+2)}\sigma_1^{-m(n+2)}\eta_1^{-m(2-p)}\Big]^{1/q} \phi_q(Dv,Q_r^\lambda)
\end{align*}
for any $\delta\in(0,1)$.
Thus \eqref{eq3.102} holds for any
\begin{equation}\label{eq:delta2}
 \delta_\gamma:=\delta \underline{\delta},\quad \text{with}\quad \delta\leq \Big(\frac{\gamma}{c_d}\Big)^\frac{1}{\beta}\Big( \underline{\delta}^{n+2}\sigma_1^{m(n+2)}\eta_1^{m(2-p)}\Big)^\frac{1}{q\beta}.
\end{equation}

\emph{Step 5: Determining the constant $\delta_\gamma$.} In view of conditions \eqref{eq:delta1} and \eqref{eq:delta2}, we can define
$$
\delta_\gamma=\underline{\delta} \Big( \underline{\delta}^{n+2}\sigma_1^{m(n+2)}\eta_1^{m(2-p)}\Big)^\frac{1}{q\beta} \Big( \frac{\gamma}{c_d+\tilde{c}_d}\Big)^\frac{1}{\min\{\beta,\beta_1\}},
$$
where we recall that $\underline{\delta}=(\eta_1^{(2-p)/2} \sigma_1)^{m+1}$.
The proof is now completed for the case $s>0$.

\emph{Step 6: The case $s=0$ via approximation.}
The approximation scheme introduced in \cite[Section 3.3]{kuusi2013mingione} also works perfectly here. The only differences are that $v$ is in place of $w$ and that instead of choosing $\tilde{\gamma}=2^{-(n+2)}\gamma$ as in \cite[Eq. (3.86)]{kuusi2013mingione}, we now take $\tilde{\gamma}=2^{-(n+2)/q}\gamma$. We also remark that for a strongly convergent sequence in $L^p$ for some $p\geq 1$, the $L^q$-mean oscillation of the sequence also converges for any $q\in(0,1)$. Indeed, by the triangle inequality, we have
\begin{align*}
    &(\phi_q(Dv_\epsilon,Q_\rho^\lambda))^q\leq \fint_{Q_\rho^\lambda}|Dv_\epsilon-\mathbf{m}(Dv,Q_\rho^\lambda)|^q\,dxdt \\&\leq (\phi_q(Dv,Q_\rho^\lambda))^q+\fint_{Q_\rho^\lambda}|Dv_\epsilon-Dv|^q\,dxdt.
\end{align*}
Switching $Dv_\epsilon$ with $Dv$, we know that
$$
|(\phi_q(Dv_\epsilon,Q_\rho^\lambda))^q-(\phi_q(Dv,Q_\rho^\lambda))^q|\leq \fint_{Q_\rho^\lambda}|Dv_\epsilon-Dv|^q\,dxdt.
$$
Therefore if $Dv_\epsilon \to Dv$ strongly in $L^p(Q_\rho^\lambda)$, then by H\"older's inequality and the previous estimate, we have
$\phi_q(Dv_\epsilon,Q_\rho^\lambda)\to \phi_q(Dv,Q_\rho^\lambda)$.
\end{proof}

\begin{proof}[Proof of Theorem \ref{thm:iter2}]
The proof also closely follow that of \cite[Theorem 3.5]{kuusi2013mingione} and therefore we only give an outline and indicate the necessary modifications.
As in the proof of Theorem \ref{thm:iter1}, we assume $s>0$ and use the singular iteration. The case $s=0$ follows by applying the same approximation scheme as in the proof of Theorem \ref{thm:iter1}. We still define $\tilde{m}$ as in \eqref{def:tm}. However, in this time since \eqref{eq:3.501} is no longer in force, the iteration may stop at $\tilde{m}$ either because the nonsingular alternative occurs or because the upper bound
\begin{equation*}
      s+\sup_{Q_{R_{\tilde{m}}}^{\lambda_{\tilde{m}}}} \|Dv\|\leq A \lambda_{\tilde{m}}
\end{equation*}
does not hold.
In the former case, we proceed as in Theorem \ref{thm:iter1} and we can take $\delta_\gamma$ as in \eqref{eq:delta2}.
In the latter case, exactly as in Step 4 of the proof of \cite[Theorem 3.5]{kuusi2013mingione}, we can show that
$$
\sup_{Q_{R_{\tilde{m}}}^{\lambda_{\tilde{m}}}} \|Dv\|\leq A \lambda_{\tilde{m}}\quad \text{and} \quad (\eta_1-\eta)A\lambda_{\tilde{m}}\leq s\leq \frac{A}{\eta_1} \lambda_{\tilde{m}}.
$$
Applying Proposition \ref{prop:iter2} in ${Q_{R_{\tilde{m}}}^{\lambda_{\tilde{m}}}}$ with $\xi=(\eta_1-\eta)A$ and $\xi_1=1/\eta_1$, we obtain
\begin{equation}\label{eq3.511}
     \phi_q(Dv, Q_{\delta R_{\tilde{m}}}^{\lambda_{\tilde{m}}})\leq \tilde{c}_d' \delta^{\beta_1'} \phi_q(Dv,Q_{R_{\tilde{m}}}^{\lambda_{\tilde{m}}})
\end{equation}
for any $\delta\in(0,1)$, where the constant $\tilde{c}_d'\geq 1$ depends on $n$, $p$, $\nu$, $L$, $A$, $q$ and the constant $\beta_1'\in(0,1)$ depends only on $n$, $p$, $\nu$, $L$, $A$. With the estimate \eqref{eq3.511} in place of \eqref{eq3.141}, we can proceed as in Theorem \ref{thm:iter1}, Step 4. In this case, \eqref{eq3.202} holds for any
\begin{equation*}
 \delta_\gamma:=\delta \underline{\delta},\quad \text{with}\quad \delta\leq \Big(\frac{\gamma}{\tilde{c}_d'}\Big)^\frac{1}{\beta_1'}\Big( \underline{\delta}^{n+2}\sigma_1^{m(n+2)}\eta_1^{m(2-p)}\Big)^\frac{1}{q\beta_1
 '} .
\end{equation*}
Finally, to ensure the above inequality and \eqref{eq:delta2}, we can choose
\begin{equation}\label{delta3}
    \delta_\gamma:=\underline{\delta} \Big( \underline{\delta}^{n+2}\sigma_1^{m(n+2)}\eta_1^{m(2-p)}\Big)^\frac{1}
    {\min\{\beta,\beta_1'\}q} \Big( \frac{\gamma}{c_d+\tilde{c}_d'}\Big)^\frac{1}{\min\{\beta,\beta_1'\}},
\end{equation}
where $\underline{\delta}=(\eta_1^{(2-p)/2} \sigma_1)^{m+1}$.
Since the constants $\sigma_1$, $\eta_1$, $\beta$, $\beta_1'$ depend only on $n$, $p$, $\nu$, $L$, $A$, and the constants $c_d$, $\tilde{c}_d'$ depend only on $n$, $p$, $\nu$, $L$, $A$, $q$, we know that the only constant depending on $B$ in \eqref{delta3} is $m$.
From the characterization of $m$ described in \eqref{ineq:stopping}, we can take $\delta_\gamma$ in the form of \eqref{eq3.203} by further decreasing the constant in \eqref{delta3}. More specifically, we can take
$$
\alpha_1:=\frac{\min\{\beta,\beta_1'\}q}{(n+5)+(2n+5)
\log(\sigma_1)/\log(\eta_1)},
\quad \alpha_2:=\min\{\beta,\beta_1'\},$$
and
$$
c(A):=(2A)^{1/\alpha_1}\eta_1^{p-2}\sigma_1^{-2} \Big(\frac{c_d+\tilde{c}_d'}{\eta_1^{2(n+3)/q}\sigma_1^{3(n+2)/q}}\Big)^\frac{1}{\alpha_2}.
$$
The proof is now completed.
\end{proof}

\section{Pointwise gradient estimates}\label{sec4}
This section is devoted to the proofs of the pointwise gradient estimates. We derive some comparison results in Section \ref{sec4.1} and give the proofs of Theorem \ref{thm:r1}--Corollary \ref{cor:1} in Section \ref{sec4.2}.
\subsection{Comparison results}\label{sec4.1}

Let $\rho,\lambda>0$ and $Q_\rho^\lambda:=Q_{\rho}^\lambda(x_0,t_0)\subset\Omega_T$ be a parabolic cylinder.
We consider the unique solution $w\in C^0([t_0-\lambda^{2-p}\rho^2, t_0);\, L^2(B_\rho(x_0)))\cap L^p(t_0-\lambda^{2-p}\rho^2, t_0;\, W^{1,p}(B_\rho(x_0)))$ to the Cauchy-Dirichlet problem:
\begin{equation}\label{eqa:w}
    \left\{
\begin{aligned}
     w_t-\div(a(x,t,D w)) =&  0 \quad \text{in} \quad Q_{\rho}^{\lambda},  &\\
     w =&  u \quad \text{on}\quad \partial_{\text{par}} Q_{\rho}^{\lambda}. &\\
\end{aligned}
\right.
\end{equation}
We have the following comparison estimate between $u$ and $w$ from \cite[Lemma 3.1]{park2020regularity}.
\begin{lemma}\label{lem:u-w}
Let $w$ be a solution to \eqref{eqa:w} under the assumptions \eqref{ineq:growth}--\eqref{dini} and assume that $p\in (\frac{3n+2}{2n+2},2-\frac{1}{n+1}]$. Then there exists a constant $c_3=c_3(n,p,\nu,L,q)\geq 1$ such that
    \begin{align*}
         &\left(\fint_{Q_{\rho}^{\lambda}}|D u-D w|^{q}\,dxdt \right)^{1/q}\\
         &\leq c_3\left[\frac{|\mu|(Q_{\rho}^{\lambda})}{|Q_{\rho}^{\lambda}|^\frac{n+1}{n+2}} \right]^\frac{n+2}{(n+1)p-n}+c_3\left[\frac{|\mu|(Q_{\rho}^{\lambda})}{|Q_{\rho}^{\lambda}|^\frac{n+1}{n+2}} \right]\Big(\fint_{Q_{\rho}^{\lambda}}(|D u|+s)^q\,dxdt\Big)^\frac{(2-p)(n+1)}{q(n+2)}
    \end{align*}
for any constant $q$ such that $\frac{n+2}{2(n+1)}<q<p-\frac{n}{n+1}\leq 1$.
\end{lemma}
We remark that only the case $s=0$ was considered in \cite[Lemma 3.2]{park2020regularity}, but their proof also works in the case $s>0$ with $V_s(\xi):=(|\xi|^2+s^2)^\frac{p-2}{4} \xi$ in place of $V(\xi):=|\xi|^\frac{p-2}{2} \xi$ for $\xi\in \R^n$.

We also have a reverse H\"older type inequality for $Dw$.
\begin{theorem}\label{rev}
Let $\lambda>0$ and $w$ be a solution to \eqref{eqa:w} under the assumptions \eqref{ineq:growth} and \eqref{ineq:elliptic} and assume that $\max\{1,\frac{2n}{n+2}\}<p<2$. Then there exists a constant $c_4=c_4(n,p,\nu,L,q)\geq 1$ such that
\begin{equation}\label{eq:rev}
\fint_{\frac{1}{2}Q_{\rho}^{\lambda}}(|D w|+s)^p\,dxdt \leq c_4\lambda^p+c_4s^p+c_4\lambda^\frac{n(p-2)(p-q)}{n(p-2)+2q}\Big(\fint_{Q_{\rho}^{\lambda}}(|D w|+s)^q\,dxdt\Big)^\frac{n(p-2)+2p}{n(p-2)+2q}
\end{equation}
holds for every $q\in(\frac{n(2-p)}{2},\frac{2n}{n+2})$ when $n\geq 2$, and that
\begin{equation}\label{eq:rev2}
\fint_{\frac{1}{2}Q_{\rho}^{\lambda}}(|D w|+s)^p\,dxdt \leq c_4\lambda^p+c_4s^p+c_4\lambda^\frac{(p-2)(p-q)}{p+q-2}\Big(\fint_{Q_{\rho}^{\lambda}}(|D w|+s)^q\,dxdt\Big)^\frac{2p-2}{p+q-2}
\end{equation}
holds for every $q\in(2-p,1)$ when $n=1$.
\end{theorem}
\begin{proof}
This type of estimates can be deduced from higher integrability results as in \cite[Theorem 1]{bogelein2008higher}  and  standard interpolation and iteration arguments as in \cite[Remark 6.12]{giusti2003direct}. See also \cite[Lemma 3.3]{park2020regularity}. For completeness, we give a direct proof below.
First, we note that it suffices to show \eqref{eq:rev} and \eqref{eq:rev2} for the special case when $\lambda=\rho=1,\,(x_0,t_0)=0$ by using a standard rescaling argument. Indeed, we can define $$
w_1(x,t):=(\lambda\rho)^{-1}w(x_0+\rho x, t_0+\lambda^{2-p}\rho^2 t)$$
and
$$a_1(x,t,\xi):=\lambda^{1-p} a(x_0+\rho x, t_0+\lambda^{2-p}\rho^2 t, \lambda \xi).$$
Then $w_1$ solves
$$\partial_t w_1 -\text{div} (a_1(x,t,Dw_1))=0 \quad \text{in} \quad Q^1_1(0)\equiv B_1(0)\times (-1,0)$$
and $a_1$ satisfies the assumptions \eqref{ineq:growth} and \eqref{ineq:elliptic} with $s_1:=\lambda^{-1}s$ in place of $s$. We can see that if \eqref{eq:rev} (or \eqref{eq:rev2}) holds for $w_1$ in $Q_1^1(0)$ with $s_1:=\lambda^{-1}s$ in place of $s$, then by rescaling back to $w$, \eqref{eq:rev} (or \eqref{eq:rev2}) also holds for $w$ in $Q^\lambda_\rho$.

Next, we prove \eqref{eq:rev} and \eqref{eq:rev2} for the case when $\lambda=\rho=1,\,(x_0,t_0)=0$.
For simplicity, in this proof,  we denote $B:=B_1(0)$, $Q:=Q^1_1(0)\equiv B(0,1)\times (-1,0)$, and $\frac{1}{2}Q:=\frac{1}{2}Q^1_1(0)\equiv B_{\frac{1}{2}}(0)\times (-\frac{1}{4},0)$.
First we take $\zeta_1\in C^\infty_0(B)$ such that $0\leq \zeta_1 \leq 1$ in $B$, $\zeta_1=1$ in $\frac{1}{2} B$ and $|D\zeta_1|\leq 4$ in $B$. We then define $\zeta:=\zeta_1/(\int_{B} \zeta_1(x)\, dx)$ and therefore $\int_{B} \zeta(x) \,dx=1$. Now we define
$\bar{w}(t):=\int_{B} w(x,t)\zeta(x)\,dx$.  If $n\geq 2$, it follows from the Sobolev-Poinc\'are inequality that
\begin{equation}\label{so-po}
    \int_{B} |w(x,t)-\bar{w}(t)|^2\,dx \leq c \Big(\int_{B} |Dw(x,t)|^\frac{2n}{n+2}\,dx\Big)^\frac{n}{n+2}
\end{equation}
holds for some constant $c=c(n)$.  On the other hand, if $n=1$, then we have
\begin{equation}\label{so-po2}
    \int_{B} |w(x,t)-\bar{w}(t)|^2\,dx \leq |B|\sup_{B}|w-\bar{w}(t)|^2\leq c \Big(\int_{B} |Dw(x,t)|\,dx\Big)^2.
\end{equation}

We now proceed with the case when $n\geq 2$ and we will indicate necessary modifications for the case when $n=1$ at the end of our proof.
By \eqref{eqa:w}, we know that
\begin{equation*}
\frac{d}{dt} \bar{w}(t)= -\int_{B} a(x,t,Dw)D\zeta \,dx:=g(t)
\end{equation*}
holds in distribution sense and therefore the equation \eqref{eqa:w} also reads
\begin{equation}\label{eqa:w2}
    (w-\bar{w})_t-\text{div}( a(x,t,Dw))+g(t)=0 \quad \text{in} \quad Q.
\end{equation}
Note that by \eqref{ineq:growth}, we have
\begin{equation}\label{eq:intg}
\int_{-1}^0|g(t)|dt\leq c\int_{Q} (|Dw|^2+s^2)^\frac{p-1}{2}\,dxdt
\end{equation}
for some constant $c=c(n,p,L)$.
Next, we choose a nondecreasing smooth function $\eta:\R\to [0,1]$ satisfying $\eta\equiv 1$ on $[-1/4,\infty)$, $\eta\equiv 0$ on $(-\infty,-1]$ and $|\eta'|\leq 2$ in $\R$. We then take $\varphi$ as the product of $\zeta$ and $\eta$, namely, $\varphi(x,t)=\zeta(x)\eta(t)$ for every $(x,t)\in Q$. We also choose
$$ \tilde{p}:=\max\big\{\frac{2p}{2-p},\frac{n+2}{n}\big\}. $$
Now formally we test the equation \eqref{eqa:w2} with $\psi:=(w-\bar{w})\varphi^{\tilde{p}} {1}_{(-\infty,\tau]}(t)$ for $\tau\in (-1,0)$ and get
\begin{equation}\label{eq:test:rev}
\begin{aligned}
    0&= \frac{1}{2} \int_B (w(x,\tau)-\bar{w}(\tau))^2 (\varphi(x,\tau))^{\tilde{p}}\,dx- \frac{\tilde{p}}{2} \int_{B\times (-1,\tau]}(w-\bar{w})^2 \varphi^{\tilde{p}-1} \varphi_t \,dxdt
    \\&\quad+\tilde{p} \int_{B\times (-1,\tau]} a(x,t,Dw) (w-\bar{w})\varphi^{\tilde{p}-1} D\varphi \,dxdt
    \\&\quad+ \int_{B\times (-1,\tau]} a(x,t,Dw) Dw\,\varphi^{\tilde{p}} \,dxdt+\int_{B\times (-1,\tau]} (w-\bar{w})\varphi^{\tilde{p}}g(t)\,dxdt
    \\&:=\RN{1}+\RN{2}+\RN{3}+\RN{4}+\RN{5}.
\end{aligned}
\end{equation}
Note that the computations above can be justified in a standard way using Steklov averages as in \cite{dibenedetto1993degenerate}. Then we estimate the terms in \eqref{eq:test:rev}. By H\"older's inequality, Young's inequality with conjugate exponents $(\tilde{p},\,\tilde{p}/(\tilde{p}-1))$, and \eqref{so-po}, we have
\begin{equation}\label{eq:rn2}
\begin{aligned}
    &|\RN{2}|\leq \frac{\tilde{p}}{2} \int_{-1}^{\tau} \Big(\int_B (w-\bar{w})^2 \varphi^{\tilde{p}} \,dx\Big)^\frac{\tilde{p}-1}{\tilde{p}}\Big(\int_B (w-\bar{w})^2 \varphi_t^{\tilde{p}}\,dx\Big)^\frac{1}{\tilde{p}} dt
   \\&\leq \frac{\tilde{p}}{2}\sup_{t\in(-1,0)} \Big(\int_B (w-\bar{w}(t))^2 \varphi^{\tilde{p}} \,dx\Big)^\frac{\tilde{p}-1}{\tilde{p}}\int_{-1}^0
    \Big(\int_B (w-\bar{w})^2 \varphi_t^{\tilde{p}}\,dx\Big)^\frac{1}{\tilde{p}}dt
    \\&\leq \frac{1}{8} \sup_{t\in(-1,0)} \int_B (w-\bar{w}(t))^2 \varphi^{\tilde{p}} \,dx +c\, \Big(\int_{-1}^0
    \Big(\int_B (w-\bar{w})^2 \,dx\Big)^\frac{1}{\tilde{p}}dt\Big)^{\tilde{p}}
    \\&\leq \frac{1}{8} \sup_{t\in(-1,0)} \int_B (w-\bar{w}(t))^2 \varphi^{\tilde{p}} \,dx +c \, \Big(\int_{-1}^0
    \Big(\int_B |Dw|^\frac{2n}{n+2} \,dx\Big)^\frac{n+2}{n\tilde{p}}dt\Big)^{\tilde{p}}
    \\&\leq \frac{1}{8} \sup_{t\in(-1,0)} \int_B (w-\bar{w}(t))^2 \varphi^{\tilde{p}} \,dx +c \, \Big(\int_Q
    |Dw|^\frac{2n}{n+2} \,dxdt\Big)^{\frac{n+2}{n}}
\end{aligned}
\end{equation}
for some constant $c=c(n,p)$. Here we also used the fact that $\frac{n+2}{n\tilde{p}}\leq 1$ in the last line.

By \eqref{ineq:growth}, Young's inequality with exponents $(p/(p-1),\,p)$ and $(2/p, \,2/(2-p)$, and the fact that $\frac{2}{p}(\tilde{p}-p)\geq\tilde{p}$, we have
\begin{equation}\label{eq:rn3}
\begin{aligned}
    &|\RN{3}|\leq \tilde{p}L\int_{B\times(-1,\tau]} (s^2+|Dw|^2)^\frac{p-1}{2} |w-\bar{w}|\varphi^{\tilde{p}-1} |D\varphi|\,dxdt
    \\&\leq \frac{\nu}{2} \int_{B\times(-1,\tau]} (s^2+|Dw|^2)^\frac{p}{2} \varphi^{\tilde{p}}\,dxdt+c \int_{B\times(-1,\tau]} \varphi^{\tilde{p}-p} |w-\bar{w}|^p |D\varphi|^p \,dxdt
    \\&\leq \frac{\nu}{2} \int_{B\times(-1,\tau]} (s^2+|Dw|^2)^\frac{p}{2} \varphi^{\tilde{p}}\,dxdt+\frac{1}{8} \int_{B\times(-1,\tau]} |w-\bar{w}|^2 \varphi^{\tilde{p}}\,dxdt
    \\ &\quad+c \int_{B\times(-1,\tau]}|D\varphi|^\frac{2p}{2-p}\,dxdt
    \\&\leq \frac{\nu}{2} \int_{B\times(-1,\tau]} |Dw|^p\varphi^{\tilde{p}}\,dxdt +\frac{1}{8} \sup_{t\in(-1,0)} \int_B (w-\bar{w}(t))^2 \varphi^{\tilde{p}} \,dx + c(1+s^p)
\end{aligned}
\end{equation}
for some constant $c=c(n,p,\nu,L)$. Here $\nu$ is the constant in \eqref{ineq:elliptic}.

By \eqref{ineq:elliptic}, we know that $a(x,t,\xi)\xi\geq \nu (s^2+|\xi|^2)^\frac{p-2}{2} |\xi|^2\geq \nu (\xi^p-s^p)$ for every $\xi\in\R^n$ and therefore
\begin{equation}\label{eq:rn4}
    \RN{4}\geq \nu \int_{B\times(-1,\tau]} |Dw|^p \varphi^{\tilde{p}}\,dxdt-c s^p
\end{equation}
for some constant $c=c(n)$.

Finally, by H\"older's inequality, Young's inequality with conjugate exponents $2$ and $2$, and the estimate \eqref{eq:intg}, we obtain
\begin{equation}\label{eq:rn5}
\begin{aligned}
    &|\RN{5}|\leq \int_{-1}^\tau \Big(\int_B |w-\bar{w}|^2 \varphi^{\tilde{p}}\,dx\Big)^\frac{1}{2} \Big(\int_B |g(t)|^2 \varphi^{\tilde{p}} \,dx\Big)^\frac{1}{2} dt
    \\&\leq c \sup_{t\in(-1,0)} \Big(\int_B (w-\bar{w}(t))^2 \varphi^{\tilde{p}} \,dx\Big)^{\frac{1}{2}} \int_{-1}^\tau |g(t)|dt
    \\&\leq \frac{1}{8} \sup_{t\in(-1,0)} \int_B (w-\bar{w}(t))^2 \varphi^{\tilde{p}} \,dx+c\Big(\int_{Q} |Dw|^{p-1}\,dxdt\Big)^2 +c (1+s^p).
\end{aligned}
\end{equation}
for some constant $c=c(n,p,L)$. Here we also used the fact that $2(p-1)\leq p$ and therefore $s^{2(p-1)}\leq 1+s^p$ in the last inequality.
Using the estimates \eqref{eq:rn2}--\eqref{eq:rn5} in \eqref{eq:test:rev},
we get
\begin{align*}
    &\frac{1}{2} \int_B (w(x,\tau)-\bar{w}(\tau))^2 (\varphi(x,\tau))^{\tilde{p}}\,dx+\frac{\nu}{2} \int_{B\times(-1,\tau]} |Dw|^p \varphi^{\tilde{p}}\,dxdt\\
    &\leq \frac{3}{8} \sup_{t\in(-1,0)} \int_B (w-\bar{w}(t))^2 \varphi^{\tilde{p}} \,dx
    +c \, \Big(\int_Q
    |Dw|^\frac{2n}{n+2} \,dxdt\Big)^{\frac{n+2}{n}}
    \\& +c \,\Big(\int_{Q} |Dw|^{p-1}\,dxdt\Big)^2+c(1+s^p)
\end{align*}
for another constant $c=c(n,p,\nu,L)$.
By taking the supremum over $\tau\in(-1,0)$, we have
\begin{align*}
    &\frac{1}{8} \sup_{t\in(-1,0)} \int_B (w-\bar{w}(t))^2 \varphi^{\tilde{p}} \,dx
    \\&\leq c \, \Big(\int_Q
    |Dw|^\frac{2n}{n+2} \,dxdt\Big)^{\frac{n+2}{n}}
     +c \,\Big(\int_{Q} |Dw|^{p-1}\,dxdt\Big)^2+c(1+s^p),
\end{align*}
and therefore
\begin{equation}\label{eq:rev3}
\begin{aligned}
    &\int_{\frac{1}{2}Q} |Dw|^p \,dxdt \leq c\int_{B\times(-1,0)} |Dw|^p \varphi^{\tilde{p}}\,dxdt
    \\&\leq c \, \Big(\int_Q
    |Dw|^\frac{2n}{n+2} \,dxdt\Big)^{\frac{n+2}{n}}
     +c \,\Big(\int_{Q} |Dw|^{p-1}\,dxdt\Big)^2+c(1+s^p)
    \end{aligned}
\end{equation}
for some constant $c=c(n,p,\nu,L)$.
Next, we interpolate the terms in \eqref{eq:rev3}. We recall that we assume $p\in(\frac{2n}{n+2},2)$ and $q\in(\frac{n(2-p)}{2}, \frac{2n}{n+2})$ for $n\geq 2$.  Since $q<\frac{2n}{n+2}<p$, by H\"older's inequality, we have
\begin{equation}\label{interpolation1}
    \Big(\int_Q
    |Dw|^\frac{2n}{n+2} \,dxdt\Big)^{\frac{n+2}{n}}\leq \Big(\int_Q
    |Dw|^p \,dxdt\Big)^{\alpha\frac{n+2}{n}}\Big(\int_Q
    |Dw|^q \,dxdt\Big)^{(1-\alpha)\frac{n+2}{n}},
\end{equation}
where $\alpha\in(0,1)$ is a constant such that $\alpha p+(1-\alpha)q=\frac{2n}{n+2}$. Namely,
$$\alpha=(\frac{2n}{n+2}-q)/(p-q).$$ Since $q>\frac{n(2-p)}{2}$, we know that $\alpha \frac{n+2}{n}<1$. Therefore by \eqref{interpolation1} and Young's inequality with exponents $1/(\alpha \frac{n+2}{n})$ and $1/(1-\alpha \frac{n+2}{n})$, we obtain
\begin{equation}\label{interpolation1-1}
     \Big(\int_Q
    |Dw|^\frac{2n}{n+2} \,dxdt\Big)^{\frac{n+2}{n}}\leq \epsilon \int_Q
    |Dw|^p \,dxdt + c_\epsilon \Big(\int_Q
    |Dw|^q \,dxdt\Big)^\frac{n(p-2)+2p}{n(p-2)+2q}
\end{equation}
for any $\epsilon\in(0,1)$, where $c_\epsilon$ is a constant depending only on $\epsilon$, $n$, $p$, and $q$.
We then estimate the term $\Big(\int_{Q} |Dw|^{p-1}\,dxdt\Big)^2$.
Since $p-1<1\leq \frac{2n}{n+2}$, by H\"older's inequality, we have
\begin{equation}\label{interpolation4}
   \begin{aligned}
   &\Big(\int_{Q} |Dw|^{p-1}\,dxdt\Big)^2\leq c \Big(\int_{Q} |Dw|^\frac{2n}{n+2}\,dxdt\Big)^\frac{(p-1)(n+2)}{n}\\&\leq c\Big[1+ \Big(\int_Q
    |Dw|^\frac{2n}{n+2} \,dxdt\Big)^{\frac{n+2}{n}}\Big]
    \end{aligned}
\end{equation}
for some constant $c=c(n,p,q)$.
Therefore, it follows by using \eqref{interpolation1-1} and \eqref{interpolation4} in \eqref{eq:rev3}, and a standard iteration argument that
$$
\fint_{\frac{1}{2}Q}(|D w|+s)^p\,dxdt \leq c(1+s^p)+c\Big(\fint_{Q}(|D w|+s)^q\,dxdt\Big)^\frac{n(p-2)+2p}{n(p-2)+2q}
$$
holds for every $q\in(\frac{(2-p)n}{2},\frac{2n}{n+2})$ and some constant $c=c(n,p,\nu,L,q)$. Note that the last estimate is exactly \eqref{eq:rev} in the case when $\lambda=\rho=1,\,(x_0,t_0)=0$. The proof for the case when $n\geq 2$ is now completed. Finally, we briefly indicate the modifications needed for the case when $n=1$. In this case, we test the equation \eqref{eqa:w2} with $(w-\bar{w})\varphi^{\tilde{p}'} {1}_{(-\infty,\tau]}(t)$ for $\tau\in (-1,0)$, where
$$ \tilde{p}':=\max\big\{\frac{2p}{2-p},2\big\}.$$ Arguing exactly as in the case when $n\geq 2$, with the estimate \eqref{so-po2} in place of \eqref{so-po}, we obtain
\begin{equation}\label{eq:rev4}
    \int_{\frac{1}{2}Q} |Dw|^p \,dxdt \leq c \, \Big(\int_Q
    |Dw| \,dxdt\Big)^2
     +c \,\Big(\int_{Q} |Dw|^{p-1}\,dxdt\Big)^2+c(1+s^p).
\end{equation}
Finally, it follows from the last estimate and similar interpolation and iteration arguments as in the case when $n\geq 2$ that
$$
\fint_{\frac{1}{2}Q}(|D w|+s)^p\,dxdt \leq c(1+s^p)+c\Big(\fint_{Q}(|D w|+s)^q\,dxdt\Big)^\frac{2p-2}{p+q-2}
$$
holds for every $q\in(2-p,1)$ and some constant $c=c(p,\nu,L,q)$.
This proves \eqref{eq:rev2} in the case when $\lambda=\rho=1$ and $(x_0,t_0)=0$. The Theorem is now proved.
\end{proof}

\begin{remark}
It can be seen from the proof of Theorem \ref{rev} that reverse H\"older type estimates also holds for $q\in[\frac{2n}{n+2},p]$ if $n\geq 2$ and for $q\in[1,p]$ if $n=1$. Indeed, by \eqref{eq:rev3}, \eqref{eq:rev4}, H\"older's inequality, and similar rescaling arguments, there exists a constant $c=c(n,p,\nu,L,q)\geq 1$ such that
$$
\fint_{\frac{1}{2}Q_{\rho}^{\lambda}}(|D w|+s)^p\,dxdt \leq c\lambda^p+cs^p+c\lambda^{p-2}\Big(\fint_{Q_{\rho}^{\lambda}}(|D w|+s)^q\,dxdt\Big)^\frac{2}{q}
$$
holds for every $q\in[\frac{2n}{n+2},p]$ if $n\geq 2$, and for every $q\in[1,p]$ if $n=1$.
\end{remark}

Now we consider the unique solution $v\in C^0([t_0-\lambda^{2-p}\rho^2, t_0);\, L^2(B_\rho(x_0)))\cap L^p(t_0-\lambda^{2-p}\rho^2, t_0;\, W^{1,p}(B_\rho(x_0)))$ to the Cauchy-Dirichlet problem:
\begin{equation}\label{eqa:v}
    \left\{
\begin{aligned}
     v_t-\div(a(x_0,t,D v)) =&  0 \quad \text{in} \quad \frac{1}{2}Q_{\rho}^{\lambda},  &\\
     v =&  w \quad \text{on}\quad \partial_{\text{par}}\big( \frac{1}{2}Q_{\rho}^{\lambda}\big). &\\
\end{aligned}
\right.
\end{equation}
We have the following comparison result between $w$ and $v$.
\begin{lemma}\label{lem:w-v}
Let $w$ be a solution to \eqref{eqa:w} and $v$ be a solution to \eqref{eqa:v} under the assumptions \eqref{ineq:growth}--\eqref{dini}. Assume that $p\in(1,2)$. Then there exists a constant $c_5=c_5(n,p,\nu,L)\geq 1$, such that
\begin{equation}\label{eq:w-v1}
\fint_{\frac{1}{2}Q^\lambda_{\rho}} |Dw-Dv|^p \,dxdt \leq c_5 \big[\omega(\rho)\big]^p \fint_{\frac{1}{2}Q^\lambda_{\rho}}(|Dw|+s)^p \,dxdt.
\end{equation}
\end{lemma}
\begin{proof}
By \cite[Remark 4.1]{kuusi2012new}, we know that
\begin{equation}\label{w-v3}
    \fint_{\frac{1}{2}Q^\lambda_\rho} |V_s(Dw)-V_s(Dv)|^2 \,dxdt \leq c \big[\omega(\rho/2)\big]^2 \fint_{\frac{1}{2}Q^\lambda_{\rho}}(|Dw|+s)^p \,dxdt,
\end{equation}
where $V_s(\xi):=(|\xi|^2+s^2)^\frac{p-2}{4} \xi$ for $\xi\in \R^n$ and
$c>0$ is a constant depending only on $n$, $p$, $\nu$, and $L$.
Also similar to \cite[Eq. (4.11)]{kuusi2012new}, we have the following inequality
\begin{equation*}
    |Dw-Dv|\leq c(|Dw|+s)^{(2-p)/2}|V_s(Dw)-V_s(Dv)|+c|V_s(Dv)-V_s(Dw)|^{2/p},
\end{equation*}
and therefore by Young's inequality with exponents $2/(2-p)$ and $2/p$, and the fact that $\omega(\rho)\in [0,1]$ for every $\rho>0$, we obtain
\begin{equation}\label{eq:well}
    |Dw-Dv|^p\leq c\big[\omega(\rho)\big]^p(|Dw|+s)^{p}+c\big[\omega(\rho)\big]^{p-2}|V_s(Dw)-V_s(Dv)|^2.
\end{equation}
Thus \eqref{eq:w-v1} follows from \eqref{w-v3}, \eqref{eq:well} and the fact that $\omega$ is a nondecreasing function.
\end{proof}
We also have a Lipschitz estimate for $v$ from \cite[Chapter 8, Theorem 5.2']{dibenedetto1993degenerate}
\begin{theorem}\label{lip}
Let $v$ be a solution to \eqref{eqa:v} under the assumptions \eqref{ineq:growth}--\eqref{dini}. Assume that $\max\{1,\frac{2n}{n+2}\}<p<2$ and that $q\in(\frac{(2-p)n}{2},p\,]$. Then there exists a constant $c_6=c_6(n,p,\nu,L,q)\geq 1$, such that
$$
\sup_{\frac{1}{4}Q^\lambda_{\rho}} \|Dv\| \leq c_6(\lambda+s)+c_6 \lambda^{\frac{n(p-2)}{n(p-2)+2q}} \Big(\fint_{\frac{1}{2}Q^\lambda_{\rho}}(|Dv|+s)^q \,dxdt\Big)^\frac{2}{n(p-2)+2q}.
$$
\end{theorem}
In the rest of this section, we always assume $$p\in\big(p^*(n), 2-\frac{1}{n+1}\big] \quad \text{and} \quad q\in \big(\max\{\frac{n+2}{2(n+1)}, \frac{(2-p)n}{2}\}, p-\frac{n}{n+1}\big),$$
where $p^*(n)$ is defined in \eqref{pstar} so that all of the assumptions in Lemma \ref{lem:u-w}--Theorem \ref{lip} are satisfied. In particular, we have $q\in(\frac{1}{2}, 1)$ and therefore
\begin{equation}\label{jensen2}
   (a+b)^q\leq a^q+b^q, \quad  (a+b)^{1/q}\leq 2a^{1/q}+2b^{1/q}, \quad \forall\,a,b>0.
\end{equation}
\begin{lemma}\label{lem:u-v}
Assume that
\begin{equation}\label{Upper}
    \Big(\fint_{Q_{\rho}^{\lambda}}(|D u|+s)^q\,dxdt\Big)^\frac{1}{q}\leq \lambda, \quad
    \frac{|\mu|(Q_{\rho}^{\lambda})}{\rho^{n+1}} \leq \lambda.
\end{equation}
Then there exists a constant $c_7=c_7(n,p,\nu,L,q)\geq 1$, such that
\begin{equation}\label{u-v}
    \Big(\fint_{\frac{1}{2}Q^\lambda_{\rho}} |Du-Dv|^q \,dxdt\Big)^\frac{1}{q} \leq c_7\, \omega(\rho)\lambda +c_7\, \Big[ \frac{|\mu|(Q_{\rho}^{\lambda})}{\rho^{n+1}}\Big].
\end{equation}
\end{lemma}
\begin{proof}
By \eqref{Upper} and Lemma \ref{lem:u-w}, we can argue as in the proof of \cite[Corollary 4.4]{kuusi2013mingione} and get
\begin{equation}\label{u-w2}
    \fint_{Q^\lambda_{\rho}} |Du-Dw|^q \,dxdt \leq c\, \Big[ \frac{|\mu|(Q_{\rho}^{\lambda})}{\rho^{n+1}}\Big]^q.
\end{equation}
Using \eqref{Upper}, \eqref{u-w2}, and the triangle inequality, we have
\begin{equation}\label{bound:w}
  \fint_{Q_{\rho}^{\lambda}}(|D w|+s)^q\,dxdt\leq c \lambda^q.
\end{equation}
From \eqref{Upper} we also have $s\leq \lambda$, which together with \eqref{bound:w} and Theorem \ref{rev} implies
$$
    \fint_{\frac{1}{2}Q_{\rho}^{\lambda}}(|D w|+s)^p\,dxdt\leq c \lambda^p.
$$
Thus, Lemma \ref{lem:w-v}, the last inequality, and H\"older's inequality imply
\begin{equation}\label{w-v2}
\fint_{\frac{1}{2}Q^\lambda_{\rho}} |Dw-Dv|^q \,dxdt \leq c\, \omega(\rho)^q \lambda^q.
\end{equation}
The estimate \eqref{u-v} now follows using \eqref{u-w2}, \eqref{w-v2} and the triangle inequality.
\end{proof}
\begin{lemma}\label{lem:supinf}
Let $\delta, \theta\in(0,1/2)$. Assume that
\begin{equation}\label{eq3.7.1}
   \Big(\fint_{Q_{\rho}^{\lambda}}(|D u|+s)^q\,dxdt\Big)^\frac{1}{q}\leq \lambda, \quad
    \frac{|\mu|(Q_{\rho}^{\lambda})}{\rho^{n+1}} \leq \frac{\delta^\frac{n+2}{q} \theta^\frac{1}{q}}{2c_7}\lambda, \quad \omega(\rho)\leq \frac{\delta^\frac{n+2}{q} \theta^\frac{1}{q}}{2c_7},
\end{equation}
where $c_7$ is the same constant as in Lemma \ref{lem:u-v}.
Then there exists a constant $c_8=c_8(n,p,\nu,L,q)\geq 1$, such that
\begin{equation*}
    s+\sup_{\frac{1}{4}Q^{\lambda}_\rho}\|Dv\|\leq c_8\lambda.
\end{equation*}
Moreover, it holds that
\begin{equation*}
    \fint_{\delta Q^{\lambda}_{\rho}} |Du|^q \,dxdt-\theta \lambda^q \leq \fint_{\delta Q^{\lambda}_{\rho}} |Dv|^q \,dxdt\leq \sqrt{n} \big(\sup_{\delta Q^{\lambda}_\rho}\|Dv\|\big)^q.
\end{equation*}
\end{lemma}

\begin{proof}
Using \eqref{eq3.7.1} and Lemma \ref{lem:u-v}, we have
\begin{equation}\label{u-v:2}
   \fint_{\frac{1}{2}Q^\lambda_{\rho}} |Du-Dv|^q \,dxdt \leq \delta^{n+2}\theta\lambda^q.
\end{equation}
This together with the first inequality in \eqref{eq3.7.1} and the triangle inequality implies
$$
\fint_{\frac{1}{2}Q_\rho^\lambda}(|Dv|+s)^q\,dxdt\leq \fint_{\frac{1}{2}Q_\rho^\lambda}(|Du|+s)^q\,dxdt+\fint_{\frac{1}{2}Q^\lambda_{\rho}} |Du-Dv|^q \,dxdt\leq (2^{n+2}+1)\lambda^q.
$$
Thus Theorem \ref{lip} and the last inequality imply
$$
s+\sup_{\frac{1}{4}Q^\lambda_{\rho}} \|Dv\| \leq c_8\lambda
$$
for some constant $c_8=c_8(n,p,\nu,L,q)\geq 1$.
Moreover, using \eqref{u-v:2} and the triangle inequality, we also have
\begin{align*}
    \fint_{\delta Q^{\lambda}_{\rho}} |Du|^q \,dxdt &\leq \fint_{\delta Q^{\lambda}_{\rho}} |Dv|^q \,dxdt+\fint_{\delta Q^\lambda_{\rho}} |Du-Dv|^q \,dxdt\\
    &\leq \fint_{\delta Q^{\lambda}_{\rho}} |Dv|^q \,dxdt+{(2\delta)}^{-(n+2)}\fint_{ \frac{1}{2}Q^\lambda_{\rho}} |Du-Dv|^q \,dxdt\\
    &\leq \fint_{\delta Q^{\lambda}_{\rho}} |Dv|^q \,dxdt+\theta \lambda^q.
\end{align*}
The lemma is proved.
\end{proof}

\begin{lemma}\label{lem:vtou}
Let $\delta\in(0,1/4)$ and $\epsilon\in(0,1]$. Suppose that the bounds
\begin{equation*}
    \Big(\fint_{Q_{\rho}^{\lambda}}(|D u|+s)^q\,dxdt\Big)^\frac{1}{q}\leq \lambda, \quad
    \frac{|\mu|(Q_{\rho}^{\lambda})}{\rho^{n+1}} \leq \lambda
\end{equation*}
hold, and that
$Dv$ satisfies
\begin{equation}\label{eq3.8.2}
    \phi_q(Dv, \delta Q^{\lambda}_\rho)\leq \frac{\epsilon}{2^{4(n+3)}} \phi_q(Dv, \frac{1}{4}Q^{\lambda}_\rho).
\end{equation}
Then
\begin{equation}\label{eq3.8.3}
    \phi_q(Du, \delta Q^{\lambda}_\rho)\leq \frac{\epsilon}{4} \phi_q(Du, Q^{\lambda}_\rho)+c_7\delta^{-(n+2)/q}\Big[\frac{|\mu|(Q_{\rho}^{\lambda})}{\rho^{n+1}}\Big]+c_7\delta^{-(n+2)/q} \omega(\rho) \lambda,
\end{equation}
where $c_7=c_7(n,p,\nu,L,q)$ is the same constant as in Lemma \ref{lem:u-v}.
\end{lemma}
\begin{proof}
Recalling the definition of $\phi_q$ and using \eqref{eq3.8.2}, \eqref{jensen2}, and the triangle inequality, we obtain
\begin{equation}\label{eq 3.8.4}
\begin{aligned}
    &\phi_q(Du,\delta Q_\rho^\lambda)\leq \Big(\fint_{\delta Q_\rho^\lambda}|Du-\mathbf{m}(Dv,\delta Q_\rho^\lambda)|^q\,dxdt\Big)^{1/q}\\
    &\leq 2\phi_q(Dv,\delta Q_\rho^\lambda)+2 \Big(\fint_{\delta Q_\rho^\lambda}|Du-Dv|^q\,dxdt\Big)^{1/q}\\
    &\leq \varepsilon 2^{-(4n+11)} \phi_q(Dv, \frac{1}{4}Q_\rho^\lambda)
    +2 {(2\delta)}^{-(n+2)/q} \Big(\fint_{\frac{1}{2} Q_\rho^\lambda}|Du-Dv|^q\,dxdt\Big)^{1/q}.
\end{aligned}
\end{equation}
Again using the definition of $\phi_q$ and the triangle inequality, we have
\begin{equation}\label{eq3.8.5}
    \begin{aligned}
     &\phi_q(Dv, \frac{1}{4}Q_\rho^\lambda)
    \leq 2\phi_q(Du, \frac{1}{4}Q_\rho^\lambda)
    + 2 \Big(\fint_{\frac{1}{4} Q_\rho^\lambda}|Du-Dv|^q\,dxdt\Big)^{1/q}\\
    &\leq 2\cdot 4^{(n+2)/q}\phi_q(Du, Q_\rho^\lambda)
    +2\cdot2^{(n+2)/q} \Big(\fint_{\frac{1}{2} Q_\rho^\lambda}|Du-Dv|^q\,dxdt\Big)^{1/q}.
    \end{aligned}
\end{equation}
Note that $q\in (1/2,1)$.
Therefore, the estimate \eqref{eq3.8.3} follows by combining \eqref{eq 3.8.4} and \eqref{eq3.8.5} and applying Lemma \ref{lem:u-v}.
\end{proof}

\subsection{Proof of pointwise gradient estimates.}\label{sec4.2}
\begin{proof}[Proof of Theorem \ref{thm:r1}.]
First, for $\lambda>0$, we define the Lebesgue set
\begin{equation}\label{Lebesgue}
\mathcal{L}_\lambda:=\big\{(x_0,t_0)\in \Omega_T:\, \lim_{\rho\to 0} \fint_{Q^\lambda_\rho(x_0,t_0)} |Du-Du(x_0,t_0)|\,dxdt=0\big\}.
\end{equation}
Direct calculations imply that $\mathcal{L}:=\mathcal{L}_1=\mathcal{L}_\lambda$ for all $\lambda>0$. Moreover, by the Lebesgue differentiation theorem, $\Omega_T\backslash \mathcal{L}$ has zero lebesgue measure. We will prove Theorem \ref{thm:r1} for every $(x_0,t_0)\in \mathcal{L}$.

\emph{Step 1: Choices of constants and basic inequalities.}
First, we take the constant $c_8=c_8(n,p,\nu,L,q)$ in Lemma \ref{lem:supinf} and define
\begin{equation}\label{const2}
A:=c_8, \quad B:=2560000n,\quad \gamma:=2^{-4(n+3)}.
\end{equation}
We fix the constants $\delta_\gamma$ and $\alpha$ in Theorem \ref{thm:iter2} with the choices of $A$, $B$, $\gamma$ in \eqref{const2}.
We now define
$$\delta_1=\delta_\gamma/4,\quad r_i=\delta_1^i r_\lambda,\quad Q_i=Q_{r_i}^\lambda$$
for any integer $i\geq 0$.
Next, we choose $k$ as the smallest integer larger than or equal to 2 such that
\begin{equation}\label{def:k}
    4\sqrt{n}A\delta_1^{(k-1)\alpha}\leq \frac{\delta_1^{(n+2)/q}}{1600},
\end{equation}
and we define
\begin{equation}\label{def:H}
    H_1:=400\delta_1^{-(n+2)/q}, \quad H_2:=5120000c_7 \delta_1^{-(k+2)(n+2)/q}, \quad c:=\max\{H_1,H_2\},
\end{equation}
where $c_7=c_7(n,p,\nu,L,q)$ is the constant in Lemma \ref{lem:u-v}.
We then choose $R_0=R_0(n,p,\nu,L,q)\in(0,\frac{1}{2}]$ to be the largest number in $(0,\frac{1}{2}]$ such that
\begin{equation}\label{def:R0}
    \int_{0}^{2R_0} \omega(\rho) \frac{d\rho}{\rho}\leq \frac{\delta_1^{(k+2)(n+2)/q}}{5120000c_7}.
\end{equation}
The constants $c$ and $R_0$ defined in \eqref{def:H} and \eqref{def:R0} are the same constants we choose in the statement of Theorem \ref{thm:r1}.
The choice of $H_1$ in \eqref{def:H} and \eqref{eq:thm1} imply that
\begin{equation}\label{initial}
\begin{aligned}
    &\Big(\fint_{Q_{0}}(|Du|+s)^q\,dxdt\Big)^\frac{1}{q}+\delta_1^{-(n+2)/q} \phi_q(Du,Q_{0})\\&\leq (1+\delta_1^{-(n+2)/q}) \Big(\fint_{Q_{r_\lambda}^\lambda}(|Du|+s)^q\,dxdt\Big)^\frac{1}{q}  \leq \frac{\lambda}{100}
\end{aligned}
\end{equation}
and that
\begin{equation}\label{bound:s}
    s\leq \frac{\lambda}{400}.
\end{equation}
Since $\delta_1\in(0,1/4)$, by using the comparison principle for the Riemann integrals, we have
\begin{equation}\label{eq:int:sum}
\begin{aligned}
    &\int_0^{2r_\lambda} \frac{|\mu|(Q^\lambda_\rho)}{\rho^{n+1}}\frac{d\rho}{\rho}\geq \sum_{i=0}^\infty \int_{r_{i}}^{2r_i} \frac{|\mu|(Q_\rho^\lambda)}{\rho^{n+1}} \frac{d\rho}{\rho}
   \\&\geq\frac 1 {n+1}(1-\frac 1 {2^{n+1}}) \sum_{i=0}^\infty \frac{|\mu|(Q_i)}{r_i^{n+1}}\geq \delta_1^{n+2}\sum_{i=0}^\infty \frac{|\mu|(Q_i)}{r_i^{n+1}}.
\end{aligned}
\end{equation}
The estimate \eqref{eq:int:sum} together with the choice of $H_2$ in \eqref{def:H} and \eqref{eq:thm1} imply that
\begin{equation}\label{sum:mu}
    \sum_{i=0}^\infty \frac{|\mu|(Q_i)}{r_i^{n+1}}\leq \frac{\delta_1^{(k+1)(n+2)/q}}{5120000c_7}\lambda.
\end{equation}
Similarly, since we have $r_\lambda\leq R_0$, the inequality \eqref{def:R0} implies that
\begin{equation}\label{sum:om}
    \sum_{i=0}^\infty \omega(r_i)\leq  \delta_1^{-(n+2)} \int_{0}^{2r_\lambda}\omega(\rho) \frac{d\rho}{\rho}\leq\frac{\delta_1^{(k+1)(n+2)/q}}{5120000c_7}.
\end{equation}

\emph{Step 2: Exit time argument.}

For any integer $i\geq 0$, we define $$C_i:=\Big(\fint_{Q_{i}}(|Du|+s)^q\,dxdt\Big)^\frac{1}{q}+\delta_1^{-(n+2)/q} \phi_q(Du,Q_{i}).$$ Thus $C_0\leq \lambda/100$ by \eqref{initial}. We can now assume without loss of generality that there exists a exit time $i_e\geq 0$, satisfying
\begin{equation*}
    C_{i_e}\leq \frac{\lambda}{100},\qquad C_j>\frac{\lambda}{100} \quad \forall j>i_e.
\end{equation*}
In fact, if this is not true, we can always find a increasing subsequence ${j_i}$ such that $C_{j_i}\leq \lambda/100$ for every integer $i\geq 0$ and therefore
$$
|Du(x_0,t_0)|\leq\limsup_{i\to \infty} \Big(\fint_{Q_{j_i}}|Du|^q\,dxdt\Big)^{1/q}\leq \frac{\lambda}{100},
$$
since $(x_0,t_0)\in \mathcal{L}$ is a Lebesgue point.

\emph{Step 3: Estimates after the exit time.}
The core of the proof is the following iteration lemma.
\begin{lemma}\label{lem:iter}
If $i\geq i_e$ and
\begin{equation}\label{supq}
    \Big(\fint_{Q_i}(|D u|+s)^q\,dxdt\Big)^\frac{1}{q}\leq \lambda,
\end{equation}
then
\begin{equation}\label{iter1}
    \phi_q(Du, Q_{i+1})\leq \frac{1}{4} \phi_q(Du, Q_i)+c_7\delta_1^{-(n+2)/q}\Big[\frac{|\mu|(Q_{i})}{r_i^{n+1}}\Big]+c_7\delta_1^{-(n+2)/q} \omega(r_i) \lambda.
\end{equation}
\end{lemma}

\begin{proof}
Let $w$ and $v$ be the solutions to \eqref{eqa:w} and \eqref{eqa:v} respectively, with $\rho=r_i$. We then take $\rho=r_i$ and $\theta=1/1600$
in Lemma \ref{lem:supinf}. By \eqref{supq}, \eqref{sum:mu} and \eqref{sum:om}, we can apply Lemma \ref{lem:supinf} both with $\delta=\delta_1$ and with $\delta=\delta_1^k$.
We obtain
\begin{equation}\label{vsup1}
    s+\sup_{Q_{i+1}}\|Dv\|\leq s+\sup_{\frac{1}{4}Q_i} \|Dv\|\leq  c_8\lambda=A\lambda
\end{equation}
and
\begin{equation}\label{vinf1}
    \fint_{ Q_{i+k}} |Du|^q \,dxdt- \frac{\lambda^q}{1600} \leq \fint_{ Q_{i+k}} |Dv|^q \,dxdt\leq \sqrt{n} \big(\sup_{ Q_{i+k}}\|Dv\|\big)^q.
\end{equation}
Next, by Theorem \ref{thm:osc} with $\rho=r_{i+k}$ and $r=r_{i+1}$, \eqref{vsup1}, and \eqref{def:k}, we have
$$
\osc\limits_{Q_{i+k}} Dv\leq 4\sqrt{n} A \delta_1^{(k-1)\alpha}\lambda\leq \frac{\delta_1^{(n+2)/q}}{1600}\lambda.
$$
Again by \eqref{supq}, \eqref{sum:mu} and \eqref{sum:om}, we can apply Lemma \ref{lem:u-v} with $\rho=r_i$ and get
\begin{align*}
  \Big(\fint_{Q_{i+k}} |Du-Dv|^q \,dxdt \Big)^\frac{1}{q}&\leq \Big(\frac{|\frac{1}{2}Q_i|}{|Q_{i+k}|}\Big)^\frac{1}{q}\Big(\fint_{\frac{1}{2}Q_{i}} |Du-Dv|^q \,dxdt\Big)^\frac{1}{q}
  \\& \leq \frac{c_7}{2} \delta_1^{-k(n+2)/q}\omega(r_i)\lambda +\frac{c_7}{2} \delta_1^{-k(n+2)/q}\Big[ \frac{|\mu|(Q_{i})}{r_i^{n+1}}\Big]
  \\&\leq \frac{\delta_1^{(n+2)/q}}{1600}\lambda.
\end{align*}
Therefore, from the above two inequalities,
\begin{align*}
  \phi_q(Du, Q_{i+k})&\leq  2^{\frac{1}{q}-1}\phi_q(Dv, Q_{i+k})+ 2^{\frac{1}{q}-1}\Big(\fint_{Q_{i+k}} |Du-Dv|^q \,dxdt\Big)^{\frac{1}{q}}
  \\&\leq 2\;\osc\limits_{Q_{i+k}} Dv+ 2\Big(\fint_{Q_{i+k}} |Du-Dv|^q \,dxdt\Big)^{\frac{1}{q}}
  \\&\leq \frac{\delta_1^{(n+2)/q}}{400}\lambda.
\end{align*}
The last estimate and \eqref{bound:s} imply that
\begin{align*}
    C_{i+k}&=\Big(\fint_{Q_{i+k}}(|Du|+s)^q\,dxdt\Big)^\frac{1}{q}+\delta_1^{-(n+2)/q} \phi_q(Du,Q_{i+k})
    \\&\leq 2\Big(\fint_{Q_{i+k}}|Du|^q\,dxdt\Big)^\frac{1}{q}+2s+\frac{\lambda}{400}
    \\&\leq 2\Big(\fint_{Q_{i+k}}|Du|^q\,dxdt\Big)^\frac{1}{q}+\frac{3\lambda}{400}.
\end{align*}
Therefore, by \eqref{vinf1}  and the fact that $C_{i+k}>\lambda/100$, we have
\begin{equation}\label{vinf2}
  \sup_{ Q_{i+1}}\|Dv\|\ge \sup_{ Q_{i+k}}\|Dv\|\geq \big(\frac{1}{1600\sqrt{n}}\big)^\frac{1}{q} \lambda\geq \frac{\lambda}{2560000n}=\frac{\lambda}{B}.
\end{equation}
By \eqref{vsup1} and \eqref{vinf2}, we can apply Theorem \ref{thm:iter2} with $Q_r^\lambda=\frac{1}{4} Q_i$ and constants $A$, $B$, $\gamma$ chosen in \eqref{const2} and get
$$\phi_q(Dv,Q_{i+1})=\phi_q(Dv, \frac{\delta_\gamma}{4}Q_i)\leq 2^{-4(n+3)} \phi_q(Dv,\frac{1}{4}Q_{i}).$$
Finally, using \eqref{supq}, \eqref{sum:mu}, and the last inequality, we can apply Lemma \ref{lem:vtou} with $\epsilon=1$ to obtain \eqref{iter1}.
\end{proof}
\emph{Step 4: Iteration and conclusion.}
For any integer $j\geq 0$, we denote $$\Phi_j:=\phi_q(Du, Q_j),\quad \mathbf{m}_j:=\mathbf{m}(Du,Q_j).$$
Since we have
$$
(|Du(x,t)|+s)^q\leq |Du(x,t)-\mathbf{m}_j|^q+(|\mathbf{m}_j|+s)^q,$$
by taking the average over $(x,t)\in Q_j$ and then taking the $q$-th root, we obtain
\begin{equation}\label{sup1}
    \Big(\fint_{Q_{j}}(|Du|+s)^q\,dxdt\Big)^\frac{1}{q}\leq 2\Phi_j+2|\mathbf{m}_j|+2s.
\end{equation}
Here we also used the fact that $q\in(1/2, 1)$. Moreover, by \eqref{eq:mbound}, we have
\begin{equation}\label{bound:m}
  |\mathbf{m}_j|\leq 2\Phi_j+2\Big(\fint_{Q_j}|Du(x,t)|^q \,dxdt\Big)^{1/q}\leq 2C_j.
\end{equation}
Using \eqref{eq:mdiff}, we also have
\begin{equation}\label{diff1}
     |\mathbf{m}_{j+1}-\mathbf{m}_j|\leq 2\Phi_{j+1}+2\delta_1^{-(n+2)/q}\Phi_j.
\end{equation}
We now prove by induction that
\begin{equation}\label{induction}
    \Phi_{j}+|\mathbf{m}_{j}|+s\leq \frac{\lambda}{2}
\end{equation}
holds for any $j\geq i_e$.
First, by the definition of exit time $i_e$ and \eqref{bound:m}, we know that
$$
C_{i_e}\le \frac\lambda {100},\quad|\mathbf{m}_{i_e}|\leq \frac{\lambda}{50},
$$
and thus
$$\Phi_{i_e}+|\mathbf{m}_{i_e}|+s\leq \frac{\lambda}{2}.$$
Assume that \eqref{induction} holds for any $j\in \{i_e,\ldots,i\}$.
Then by \eqref{sup1} we have
$$\Big(\fint_{Q_{j}}(|Du|+s)^q\,dxdt\Big)^\frac{1}{q}\leq \lambda $$
for any $j\in \{i_e,\ldots,i\}$.
Therefore we can apply Lemma \ref{lem:iter} to get
\begin{equation}\label{iterA}
    \Phi_{j+1}\leq \frac{1}{4}\Phi_{j}+c_7\delta_1^{-(n+2)/q}\Big[\frac{|\mu|(Q_{i})}{r_j^{n+1}}\Big]+c_7\delta_1^{-(n+2)/q} \omega(r_j) \lambda
\end{equation}
for any $j\in \{i_e,\ldots,i\}$.
Thus, using \eqref{induction} with $j=i$,  \eqref{iterA}, \eqref{sum:mu} and \eqref{sum:om}, we have
 \begin{equation}\label{boundA}
      \Phi_{i+1}\leq \frac{\lambda}{8}+c_7\delta_1^{-(n+2)/q}\Big[\frac{|\mu|(Q_{i})}{r_i^{n+1}}\Big]+c_7\delta_1^{-(n+2)/q} \omega(r_i) \lambda\leq \frac{\lambda}{4}.
 \end{equation}
Summing up \eqref{iterA} in $j\in \{i_e,\ldots,i\}$, we also have
$$
\sum_{j=i_e}^{i+1} \Phi_j\leq \Phi_{i_e}+ \frac{1}{4}\sum_{j=i_e}^{i}\Phi_{j}+c_7\delta_1^{-(n+2)/q}\sum_{j=i_e}^{i}\Big[\frac{|\mu|(Q_{i})}{r_j^{n+1}}\Big]+c_7\delta_1^{-(n+2)/q} \sum_{j=i_e}^{i}\omega(r_j)\lambda
$$
and thus
\begin{equation}\label{sumA}
     \sum_{j=i_e}^{i+1} \Phi_j\leq \frac{4}{3}\Phi_{i_e}+\frac{4}{3} c_7\delta_1^{-(n+2)/q}\sum_{j=i_e}^{i}\Big[\frac{|\mu|(Q_{j})}{r_j^{n+1}}\Big]+\frac{4}{3}c_7\delta_1^{-(n+2)/q} \sum_{j=i_e}^{i}\omega(r_j)\lambda.
\end{equation}
Using \eqref{diff1}, \eqref{sumA}, \eqref{sum:mu}, and \eqref{sum:om}, we obtain
\begin{align*}
   &|\mathbf{m}_{i+1}-\mathbf{m}_{i_e}|\leq \sum_{j=i_e}^i|\mathbf{m}_{j+1}-\mathbf{m}_{j}|
   \leq 4\delta_1^{-(n+2)/q} \sum_{j=i_e}^{i+1}\Phi_j\\
   &\leq 8\delta_1^{-(n+2)/q} \Phi_{i_e}+8 c_7\delta_1^{-2(n+2)/q}\sum_{j=i_e}^{i}\Big[\frac{|\mu|(Q_{j})}{r_j^{n+1}}\Big]+8c_7\delta_1^{-2(n+2)/q} \sum_{j=i_e}^{i}\omega(r_j)\lambda\\
   &\leq 8\delta_1^{-(n+2)/q} \Phi_{i_e}+\frac{\lambda}{100}.
\end{align*}
Therefore it follows from \eqref{bound:m} and the previous inequality that
\begin{equation}\label{boundm2}
    |\mathbf{m}_{i+1}|\leq |\mathbf{m}_{i_e}|+8\delta_1^{-(n+2)/q}\Phi_{i_e}+\frac{\lambda}{100}\leq 10 C_{i_e}+\frac{\lambda}{100}\leq \frac{\lambda}{8}.
\end{equation}
By \eqref{bound:s}, \eqref{boundA} and \eqref{boundm2}, we obtain
$$
\Phi_{i+1}+|\mathbf{m}_{i+1}|+s\leq \frac{\lambda}{4}+\frac{\lambda}{8}+\frac{\lambda}{400}\leq \frac{\lambda}{2},
$$
which completes the induction.
Since we have
$$|\mathbf{m}_j-Du(x_0,t_0)|^q\leq |Du(x,t)-\mathbf{m}_j|^q+|Du(x,t)-Du(x_0,t_0)|^q,$$
by taking the average over $(x,t)\in Q_j$ and then taking the $q$-th root, we obtain
\begin{equation}\label{eq:m-du}
\begin{aligned}
  |\mathbf{m}_j-Du(x_0,t_0)|&\leq  2\phi_q(Du,Q_j)+2\Big(\fint_{Q_j}|Du-Du(x_0,t_0)|^q \,dxdt\Big)^{1/q}\\&\leq 4\Big(\fint_{Q_j}|Du-Du(x_0,t_0)|^q \,dxdt\Big)^{1/q}.
 \end{aligned}
\end{equation}
Since $(x_0,t_0)\in \mathcal{L}$ is a Lebesgue point, using \eqref{eq:m-du} and \eqref{induction}, we obtain
$$
|Du(x_0,t_0)|=\lim_{j\to\infty} |\mathbf{m}_j|\leq \frac{\lambda}{2}.
$$
The proof of Theorem \ref{thm:r1} is completed.
\end{proof}

\begin{proof}[Proof of Theorem \ref{thm:int}.]
Without loss of generality, we assume that $\mathbf{I}_1^{|\mu|}(x_0,t_0,2r)<\infty$. We consider the function
\begin{align*}
    h(\lambda)&:=\lambda-c \,\Big(\fint_{Q_{r_\lambda}^\lambda(x_0,t_0)} (|Du|+s+1)^q\,dxdt\Big)^{1/q} -c\, \int_{0}^{2r_\lambda} \frac{|\mu|(Q_\rho^\lambda(x_0,t_0))}{\rho^{n+1}} \frac{d\rho}{\rho}\\
&=\lambda-c\lambda^\frac{n(2-p)}{2q} A(\lambda)-c\lambda^\frac{(n+1)(2-p)}{2} B(\lambda),
\end{align*}
where $r_\lambda=\lambda^{(p-2)/2}r$ and
$$
A(\lambda):=\frac{1}{|Q_r(x_0,t_0)|^{1/q}} \Big(\int_{Q_{r_\lambda}^\lambda} (|Du|+s+1)^q\,dxdt\Big)^{1/q}
$$
and
$$
B(\lambda):=\int_0^{2r} \frac{|\mu|(Q^\lambda_{\lambda^{(p-2)/2}\rho}(x_0,t_0))}{\rho^{n+1}}\frac{d\rho}{\rho}.
$$
Here $c$ is the same constant as in Theorem \ref{thm:r1}.
Since $p\in(1,2)$, we have $Q^{\lambda_2}_{\lambda_2^{(p-2)/2}\rho}\subset Q^{\lambda_1}_{\lambda_1^{(p-2)/2}\rho}$ for every $\lambda_2>\lambda_1>0$ and $\rho>0$ and, therefore, $A$ and $B$ are nonincreasing functions of $\lambda$. Moreover, the functions $A$, $B$, $h$ are well defined for $\lambda\in[1,\infty)$ since $Q^\lambda_{2\lambda^{(p-2)/2}r}(x_0,t_0)\subset Q_{2r}(x_0,t_0) \subset \Omega$ for any $\lambda\in[1,\infty)$.
Clearly, $h$ is a continuous function on $[1,\infty)$ and
$h(1)\leq 0$ since $c\geq 1$ and $A(1)\geq 1$. On the other hand, since $q>\frac{n(2-p)}{2}$ and $p>\frac{2n}{n+1}$, we have
$$
\lim_{\lambda\to\infty} h(\lambda)\geq \lim_{\lambda\to\infty} \Big(\lambda-c\lambda^\frac{n(2-p)}{2q} A(1)-c\lambda^\frac{(n+1)(2-p)}{2} B(1)\Big)= \infty.
$$
Thus there exists some $\lambda\geq 1$ such that $h(\lambda)=0$ and therefore \eqref{eq:thm1} holds for such $\lambda$. Since  $\lambda\geq 1$ and $r\in(0,R_0]$, we have $r_\lambda\equiv \lambda^{(p-2)/2}r\in(0, R_0]$.
Applying Theorem \ref{thm:r1} and  using the fact that $h(\lambda)=0$, we obtain
\begin{equation}\label{eq:thm1:2}
\begin{aligned}
\lambda+|Du(x_0,t_0)|\leq 2\lambda=2c\lambda^\frac{n(2-p)}{2q} A(\lambda)+2c\lambda^\frac{(n+1)(2-p)}{2} B(\lambda).
\end{aligned}
\end{equation}
Using the fact that $A,\,B$ are nonincreasing functions and Young's inequality with conjugate exponents
$\big(\frac{2q}{n(2-p)},\frac{2q}{2q-n(2-p)}\big)$ and
$\big(\frac{2}{(n+1)(2-p)},\frac{2}{(n+1)p-2n}\big)$, we obtain
\begin{align*}
    2c\lambda^\frac{n(2-p)}{2q} A(\lambda)\leq 2c\lambda^\frac{n(2-p)}{2q} A(1)\leq \frac{\lambda}{4}+c'\,[A(1)]^\frac{2q}{2q-n(2-p)}
\end{align*}
and
\begin{equation}\label{eq:thm1:2.2}
    2c\lambda^\frac{(n+1)(2-p)}{2} B(\lambda) \leq 2c\lambda^\frac{(n+1)(2-p)}{2} B(1)\leq \frac{\lambda}{4}+c'\,[B(1)]^\frac{2}{(n+1)p-2n}.
\end{equation}
Therefore \eqref{eq:thm2} follows by using the last two inequalities and \eqref{eq:thm1:2}.
\end{proof}

\begin{proof}[Proof of Corollary \ref{cor:1}.]
Corollary \ref{cor:1} follows similarly as in the proof of Theorem \ref{thm:int}. The only difference is that we need to replace \eqref{eq:thm1:2.2} with the following estimate
$$
2c\lambda^\frac{(n+1)(2-p)}{2} B(\lambda)\leq \lambda/4+c'' \|f\|_{L^\infty}^{1/(p-1)} [\mathbf{I}_1^{\mu_0}(x_0,2r)]^{1/(p-1)},
$$
which was already proved in \cite[Corollary 1.3]{kuusi2013mingione}.
\end{proof}

\section{Gradient continuity results}\label{sec5}

\subsection{Preliminary choices of constants and the geometry}
In this section, we always assume $p\in(p^*(n),2-\frac{1}{n+1}]$, where $p^*(n)$ is defined in \eqref{pstar}. We also choose $q:=q(n,p)\in(1/2,1)$ as a fixed constant depending only on $n$ and $p$. For instance, we can take
$$q:=\frac{1}{2}\Big(\max\{\frac{n+2}{2(n+1)}, \frac{(2-p)n}{2}\}+ p-\frac{n}{n+1}\Big)\in (\frac{1}{2},1).$$
The choices of geometry in this section are essentially the same as in \cite[Section 5.1]{kuusi2013mingione}. For completeness, we shall still briefly report the choices.
First, we fix an open cylinder $\tilde{Q}\subset\subset \Omega_T$ and take another cylinder $\tilde{Q}'$ such that
$\tilde{Q}\subset\subset \tilde{Q}'\subset\subset \Omega_T$. Let $\bar{R}_1:=\text{dist}_{\text{par}}(\tilde{Q}, \partial_{\text{par}} \tilde{Q}')>0$.
Under the assumptions of Theorem \ref{thm:cty} or Theorem \ref{thm1.8}, it always holds that the Riesz potential $\mathbf{I}^{|\mu|}_1(x,t;r)$ is locally bounded in $\Omega_T$ for some $r>0$. Therefore we can apply Theorem \ref{thm:int} to obtain that $Du$ is locally bounded in $\Omega_T$
and that in particular, $Du$ is bounded in $\tilde{Q}'$.
Thus we can choose
\begin{equation*}
    M:=1+s+\sup_{\tilde{Q}'}|Du|<\infty.
\end{equation*}
Let
\begin{equation}\label{def:R1}
    \lambda_M:=M \quad \text{and} \quad R_1:=\frac{1}{4}\lambda_M^{(p-2)/2} \bar{R}_1.
\end{equation}
Then we have $Q_r^{\lambda_M}(x_0,t_0)\subset \tilde{Q}'$ whenever $(x_0,t_0)\in\tilde{Q}$ and $r\in(0, R_1]$, and therefore
\begin{equation}\label{2sup}
 s+\sup_{Q_r^{\lambda_M}(x_0,t_0)}|Du|\leq \lambda_M, \quad \forall r\in(0,R_1].
\end{equation}
By \eqref{def:R1} and \eqref{2sup}, we have
\begin{equation}\label{eq5.2.2}
    \Big(\fint_{Q_\rho^{\lambda_M}(x_0,t_0)}(|Du|+s)^q \,dxdt\Big)^{1/q}\leq M\equiv\lambda_M
\end{equation}
for any $(x_0,t_0)\in \tilde{Q}$ and any $\rho\in(0,R_1]$.
\subsection{Proof of gradient continuity results}
First, we prove the following proposition.
\begin{proposition}\label{lem:cty}
Let $\varepsilon\in (0,1]$.
Assume that the Riesz potential $\mathbf{I}^{|\mu|}_1(x,t;r)$ is locally bounded in $\Omega_T$ for some $r>0$ and that
\begin{equation}\label{eq:mulim}
\lim_{r\to 0} \frac{|\mu|(Q_r(x,t))}{r^{n+1}}=0 \quad \text{locally uniformly in } (x,t)\in \Omega_T.
\end{equation}
Then there exists constants $\alpha=\alpha(n,p,\nu,L)\in(0,1)$, $c_9=c_9(n,p,\nu,L)\geq 1$, and ${R_\epsilon}={R_\epsilon}(n,p,\nu,L, M,\mu, \epsilon)\in(0, R_1)$
such that
\begin{equation}\label{eq5.2.1}
    \phi_q(Du, Q_\rho^{\lambda_M}(x_0,t_0))<\epsilon \lambda_M
\end{equation}
holds for every $(x_0,t_0)\in \tilde{Q}$ and every $\rho\in(0,\rho_\epsilon\,]$, where
$$
\rho_\epsilon=\frac{\epsilon^{2/\alpha}}{c_9}{R_\epsilon}.
$$
More specifically, the constant $R_\epsilon$ is determined in \eqref{eq5.2.5}--\eqref{eq5.2.5.1} below.
\end{proposition}
\begin{proof}
First we take
\begin{equation}\label{eq5.2.3}
    \lambda=\lambda_M, \quad A=c_8, \quad B=\frac{400n}{\epsilon},
\quad \gamma=\frac{\epsilon}{2^{4(n+3)}},
\end{equation}
where $c_8=c_8(n,p,\nu,L)$ is the same constant as in Lemma \ref{lem:supinf}. Then we choose $\delta_\gamma=\delta_\gamma(n,p,\nu,L,\epsilon)\in(0,1/2)$ as in Theorem \ref{thm:iter2} with the choices of $A$, $B$, $\gamma$ in \eqref{eq5.2.3} and we set $\delta_1=\delta_\gamma/4$. Then by \eqref{eq3.203}, we have
\begin{equation}\label{eq5.2.4}
\delta_1=\frac{\epsilon^{2/\alpha}}{c_9}
\end{equation}
for some constants $\alpha\in(0,1)$ and $c_9\geq 1$ both depending on $n$, $p$, $\nu$, $L$.
Next, we take ${R_\epsilon}\in(0,R_1)$ such that
\begin{equation}\label{eq5.2.5}
    \omega({R_\epsilon})\leq  \frac{\delta_1^\frac{n+2}{q} \epsilon}{800c_7}
\end{equation}
and that
\begin{equation}\label{eq5.2.5.1}
    \sup_{(x_0,t_0)\in\tilde{Q}}\sup_{0<\rho\leq \lambda_M^{(2-p)/2}{R_\epsilon}} \frac{|\mu|(Q_\rho(x_0,t_0))}{\rho^{n+1}}\leq \frac{\delta_1^\frac{n+2}{q} \epsilon}{800c_7\lambda_M^{(n+1)(2-p)/2}},
\end{equation}
where $c_7=c_7(n,p,\nu,L)$ is the same constant as in Lemma \ref{lem:u-v}.
Thus we have
\begin{equation}\label{eq5.2.6}
\begin{aligned}
     \sup_{(x_0,t_0)\in\tilde{Q}}\sup_{0<\rho\leq {R_\epsilon} }\frac{|\mu|(Q_\rho^{\lambda_M}(x_0,t_0))}{\rho^{n+1}}&\leq  \sup_{(x_0,t_0)\in\tilde{Q}}\sup_{0<\rho\leq {R_\epsilon} }\frac{|\mu|(Q_{\lambda_M^{(2-p)/2}\rho}(x_0,t_0))}{\rho^{n+1}}
     \\&\leq \frac{\delta_1^\frac{n+2}{q} \epsilon}{800c_7}\leq \frac{\delta_1^\frac{n+2}{q} \epsilon}{800c_7}\lambda_M.
     \end{aligned}
\end{equation}
For $i\in\mathbb{N}$, we define
\begin{equation*}
    Q_i:=Q_{r_i}^{\lambda_M}(x_0,t_0), \quad r_i:=\delta_1^i r, \quad r\in (\delta_1 {R_\epsilon},{R_\epsilon}].
\end{equation*}
We will prove that for every $i\geq 1$, it holds that
\begin{equation}\label{eq5.2.8}
    \phi_q(Du,Q_i)<\epsilon \lambda_M.
\end{equation}
Let $i\geq 1$. We consider two different cases. First, suppose that
\begin{equation}\label{eq5.2.9}
   \Big(\fint_{Q_i} |Du|^q \,dxdt\Big)^{1/q}<\frac{\epsilon}{10}\lambda_M.
\end{equation}
In this case, the definition of $\phi_q$ implies that
$$
\phi_q(Du,Q_i)\leq \Big(\fint_{Q_i} |Du|^q \,dxdt\Big)^{1/q}<\frac{\epsilon}{10}\lambda_M
$$
and therefore \eqref{eq5.2.8} holds. On the other hand, suppose that \eqref{eq5.2.9} does not hold. Then we have
\begin{equation}\label{eq5.2.10}
   \Big(\fint_{Q_i} |Du|^q \,dxdt\Big)^{1/q}\geq\frac{\epsilon}{10}\lambda_M.
\end{equation}
Let $w$ and $v$ be the solutions defined in \eqref{eqa:w} and \eqref{eqa:v} respectively, with the choices $\rho=r_{i-1}$ and $\lambda=\lambda_M$. By \eqref{eq5.2.2}, \eqref{eq5.2.5}, and \eqref{eq5.2.6}, we can apply Lemma \ref{lem:supinf} with the parameters
$\rho=r_{i-1}$, $\lambda=\lambda_M$, $\delta=\delta_1$, and  $\theta=\epsilon^q/20$. Thus using \eqref{eq5.2.10} we obtain
$$
\frac{\lambda_M}{B}=\frac{\epsilon\lambda_M}{400n}\leq \sup_{Q_i}\|Dv\|, \quad s+\sup_{\frac{1}{4}Q_{i-1}} \|Dv\|\leq c_8\lambda_M=A\lambda_M.
$$
Recalling that $\delta_1=\delta_\gamma/4$ and applying Theorem \ref{thm:iter2} in $\frac{1}{4}Q_{i-1}=\frac{1}{4}Q_{r_{i-1}}^{\lambda_M}$, we have
\begin{equation}\label{eq5.2.11}
    \phi_q(Dv,Q_i)=\phi(Dv, \frac{\delta_\gamma}{4}Q_{i-1})\leq \frac{\epsilon}{2^{4(n+3)}}\phi_q(Dv,\frac{1}{4}Q_{i-1}).
\end{equation}
By \eqref{eq5.2.2}, \eqref{eq5.2.6}, and \eqref{eq5.2.11}, we can apply Lemma \ref{lem:vtou} and get
\begin{align*}
    \phi_q(Du,Q_i)&\leq \frac{\epsilon}{4} \phi_q(Du,Q_{i-1})+c_7\delta_1^{-(n+2)/q} \Big[\frac{|\mu|(Q_{i-1}}{r_{i-1}^{n+1}}\Big]+c_7 \delta_1^{-(n+2)/q} \omega(r_{i-1})\lambda_M\\
    &\leq \frac{\epsilon}{4} \lambda_M+\frac{\epsilon}{400} \lambda_M< \epsilon \lambda_M.
\end{align*}
The proof of \eqref{eq5.2.8} is  completed.
Now we take $\rho_\epsilon=\delta_1 {R_\epsilon}$, where $\delta_1$ has the form in \eqref{eq5.2.4}. Since for any $\rho\in(0,\rho_\epsilon]$, there exists $r\in (\delta_1 {R_\epsilon}, {R_\epsilon}]$ and an integer $k\geq 1$ such that
$\rho=\delta_1^k r$, \eqref{eq5.2.1} follows directly from \eqref{eq5.2.8}.
\end{proof}
A corollary of Proposition \ref{lem:cty} is Theorem \ref{thm1.8}.
\begin{proof} [Proof of Theorem \ref{thm1.8}.]
We are now able to determine the exact form of $R_\epsilon$ in Proposition \ref{lem:cty} for any $\epsilon\in(0,1)$ thanks to the assumption \eqref{asp5}. By \eqref{eq5.2.4}, to verify \eqref{eq5.2.5} and \eqref{eq5.2.5.1}, we need to show that
$$
 \omega(R_\epsilon)\leq  \frac{\epsilon^{\frac{2(n+2)}{q\alpha}+1} }{800c_7c_9^{(n+2)/q}},
$$
and that
$$
    \sup_{(x_0,t_0)\in\tilde{Q}}\sup_{0<\rho\leq \lambda_M^{(2-p)/2}R_\epsilon} \frac{|\mu|(Q_\rho(x_0,t_0))}{\rho^{n+1}}\leq \frac{\epsilon^{\frac{2(n+2)}{q\alpha}+1}} {800c_7 c_9^{(n+2)/q}\lambda_M^{(n+1)(2-p)/2}}.
$$
Thus using \eqref{asp5}, it is sufficient to take
$R_\epsilon\in(0,R_1)$ such that
$$
R_\epsilon\leq \Big(\frac{\epsilon^{\frac{2(n+2)}{q\alpha}+1}} {800c_D c_7 c_9^{(n+2)/q}\lambda_M^{(n+2)(2-p)/2}}\Big)^{1/\delta}
:=\frac{\epsilon^{1/\theta}}{c_{10}},
$$
where $\theta=\theta(n,p,\nu,L,\delta)\in(0,1)$ and $c_{10}=c_{10}(n,p,\nu,L,\delta,c_D,M)\geq 1$.
Now we take
$$R_\epsilon=\min\{1,R_1\}\frac{\epsilon^{1/\theta}}{c_{10}}\quad \text{and} \quad \rho_\epsilon=\min\{1,R_1\}\frac{\epsilon^{2/\alpha+1/\theta}}{c_9c_{10}}.$$ Thus we can apply Lemma \ref{lem:cty} and obtain
$$
\phi_q(Du, Q_{\rho_\epsilon}^{\lambda_M}(x_0,t_0))<\epsilon \lambda_M.
$$
Since $\epsilon$ is an arbitrary number in $(0,1)$,  the last inequality implies that
\begin{equation}\label{eq:holder5.1}
\phi_q(Du, Q_\rho^{\lambda_M}(x_0,t_0))< c\rho^{\beta}
\end{equation}
holds for every $(x_0,t_0)\in \tilde{Q}$ and every $\rho\in (0,R_2]$, where $R_2:=\min\{1,R_1\}/(c_9 c_{10})\in (0,1)$, $c>0$ is a constant depending on $n$, $p$, $\nu$, $L$, $\delta$, $c_D$, $M$, and $R_1$, and
$$\beta:=\frac{1}{\frac{2}{\alpha}+\frac{1}{\theta}}\in(0,1)$$
depends only on $n$, $p$, $\nu$, $L$, and $\delta$.

We are ready to prove $Du\in C^{0,{\beta}}(\tilde{Q})$ using \eqref{eq:holder5.1}. Let $(x_1,t_1),\, (x_2,t_2)\in \tilde{Q}\cap \mathcal{L}$, such that $\rho:=|(x_1,t_1)-(x_2,t_2)|_{\text{par}}\leq R_2/2$. Here $\mathcal{L}\equiv \mathcal{L}_{\lambda_M}$ is the set of Lebesgue points of $Du$ defined in \eqref{Lebesgue}. Without loss of generality, we assume that $t_1\leq t_2$.
Arguing exactly as in \eqref{eq:m-du}, we know that
\begin{equation}\label{limit5.1}
\lim_{r\to 0} |\mathbf{m}(Q_{r}^{\lambda_M}(x_1,t_1)-Du(x_1,t_1)|=0.
\end{equation}
Using \eqref{eq:mdiff}, \eqref{limit5.1}, \eqref{eq:holder5.1}, and the fact that $q\in(1/2, 1)$, we obtain
\begin{equation}\label{sum5.1.1}
\begin{aligned}
    &|\mathbf{m}(Du, Q_\rho^{\lambda_M}(x_1,t_1))-Du(x_1,t_1)|\\&\leq \sum_{j=0}^\infty |\mathbf{m}(Du, Q_{2^{-j}\rho}^{\lambda_M}(x_1,t_1))-\mathbf{m}(Du, Q_{2^{-(j+1)}\rho}^{\lambda_M}(x_1,t_1))|\\&\leq 2^{2n+6} \sum_{j=0}^\infty \phi_{q}(Du, Q_{2^{-j}\rho}^{\lambda_M}(x_1,t_1))\leq c\,2^{2n+6} \sum_{j=0}^\infty (2^{-j}\rho)^{\beta}\leq c' \rho^{\beta}.
\end{aligned}
\end{equation}
Similarly, we have
\begin{equation}\label{sum5.1.2}
|\mathbf{m}(Du, Q_{2\rho}^{\lambda_M}(x_2,t_2))-Du(x_2,t_2)|\leq c'' \rho^{\beta}.
\end{equation}
By \eqref{sum5.1.1}, \eqref{sum5.1.2} and the triangle inequality, it holds that
\begin{equation}\label{eq:5.71}
\begin{aligned}
    &|Du(x_1,t_1)-Du(x_2,t_2)|
    \\&\leq |Du(x_1,t_1)-\mathbf{m}(Du, Q_\rho^{\lambda_M}(x_1,t_1))|
    +|\mathbf{m}(Du, Q_{2\rho}^{\lambda_M}(x_2,t_2))-Du(x_2,t_2)|
    \\&\quad+|\mathbf{m}(Du, Q_\rho^{\lambda_M}(x_1,t_1))-\mathbf{m}(Du, Q_{2\rho}^{\lambda_M}(x_2,t_2))|
    \\ &\leq c\rho^{\beta}+|\mathbf{m}(Du, Q_\rho^{\lambda_M}(x_1,t_1))-\mathbf{m}(Du, Q_{2\rho}^{\lambda_M}(x_2,t_2))|.
\end{aligned}
\end{equation}
Recalling the definition of parabolic distance and the assumption that $t_1\leq t_2$, we know that $Q_\rho^{\lambda_M}(x_1,t_1)\subset Q_{2\rho}^{\lambda_M}(x_2,t_2)$ and therefore by \eqref{eq:mdiff}, \eqref{eq:holder5.1} and the fact that $q\in(\frac{1}{2}, 1)$, we obtain
\begin{equation}\label{eq5.72}
\begin{aligned}
    &|\mathbf{m}(Du, Q_\rho^{\lambda_M}(x_1,t_1))-\mathbf{m}(Du, Q_{2\rho}^{\lambda_M}(x_2,t_2))|
    \\&\leq 2\phi_q(Du,Q_\rho^{\lambda_M}(x_1,t_1))+2^{2n+5} \phi_q(Du, Q_{2\rho}^{\lambda_M}(x_2,t_2))\leq c \rho^{\beta}.
\end{aligned}
\end{equation}
Combining \eqref{eq:5.71}, \eqref{eq5.72} and using the fact that $\rho=|(x_1,t_1)-(x_2,t_2)|_{\text{par}}$, we have
$$
|Du(x_1,t_1)-Du(x_2,t_2)|\leq c \rho^{\beta}= c\,|(x_1,t_1)-(x_2,t_2)|_{\text{par}}^{\beta},
$$
where $c$ is a constant depending only on $n$, $p$, $\nu$, $L$, $\delta$, $c_D$, $M$, and $R_1$. Recall that by Theorem \ref{thm:int}, $Du$ is bounded in $\tilde{Q}$. Therefore the last estimate implies that $Du\in C^{0,\beta}(\tilde{Q})$ since $|\tilde{Q}\backslash (\mathcal{L}\cap \tilde{Q})|=0$. The proof is now completed.
\end{proof}

Next, we turn to the proofs of Theorem \ref{thm:cty} and its corollaries. We start with the following proposition.
\begin{proposition}\label{lem:conv}
Let $\varepsilon\in (0,1]$.
Under the same assumptions as in Theorem \ref{thm:cty}, there exists a radius $R_\epsilon'\in(0, R_1)$ depending only on $n$, $p$, $\nu$, $L$, $\delta$, $c_D$, $M$, and $R_1$, such that
\begin{equation}\label{eq 5.8.1}
    |\mathbf{m}(Du,Q_{\rho_1}^{\lambda_M}(x_0,t_0))-\mathbf{m}(Du,Q_{\rho_2}^{\lambda_M}(x_0,t_0))|\leq 3\epsilon \lambda_M
\end{equation}
holds for every $\rho_1,\,\rho_2\in (0,R_\epsilon']$ and
$(x_0,t_0)\in \tilde{Q}$.
\end{proposition}
\begin{proof}
We still use an exit time argument similar to the proof of \cite[Theorem 1.6]{kuusi2013mingione}.

\emph{Step 1: Choices of constants.}
First, we take
\begin{equation}\label{eq5.8.2}
     \lambda=\lambda_M, \quad A=c_8, \quad B=\frac{400n}{\epsilon},
\quad \gamma={2^{-4(n+3)}},
\end{equation}
where $c_8=c_8(n,p,\nu,L)$ is the same constant as in Lemma \ref{lem:supinf}. Then we choose $\delta_\gamma=\delta_\gamma(n,p,\nu,L,\epsilon)\in(0,1/2)$ as in Theorem \ref{thm:iter2} with the choices of $A$, $B$, $\gamma$ in \eqref{eq5.8.2} and we set $\delta_1=\delta_\gamma/4$.

Next, We take ${R_\epsilon'}\in(0,R_1)$ depending only on $n$, $p$, $\nu$, $L$, $\delta$, $c_D$, $M$, and $R_1$, such that
\begin{equation}\label{eq5.8.3.0}
    \int_{0}^{2R_\epsilon'} \omega(\rho) \frac{d\rho}{\rho} \leq \frac{\delta_1^\frac{4(n+2)}{q} \epsilon}{800c_7},
\end{equation}

\begin{equation}\label{eq5.8.3}
    \omega({R_\epsilon'})\leq  \frac{\delta_1^\frac{n+2}{q} \epsilon}{800c_7},
\end{equation}
\begin{equation}\label{eq5.8.4}
    \sup_{(x_0,t_0)\in\tilde{Q}}\int_{0}^{2 \lambda_M^{(2-p)/2}{R_\epsilon'}} \frac{|\mu|(Q_\rho(x_0,t_0))}{\rho^{n+1}}\frac{d\rho}{\rho}\leq \frac{\delta_1^\frac{4(n+2)}{q} \epsilon}{800c_7\lambda_M^{(n+1)(2-p)/2}},
\end{equation}

\begin{equation}\label{eq5.8.5}
    \sup_{(x_0,t_0)\in\tilde{Q}}\sup_{0<\rho\leq \lambda_M^{(2-p)/2}{R_\epsilon'}} \frac{|\mu|(Q_\rho(x_0,t_0))}{\rho^{n+1}}\leq \frac{\delta_1^\frac{n+2}{q} \epsilon}{800c_7\lambda_M^{(n+1)(2-p)/2}},
\end{equation}
and
\begin{equation}\label{eq5.8.6}
    \sup_{(x_0,t_0)\in\tilde{Q}}\sup_{0<\rho\leq {R_\epsilon'}} \phi_q(Du, Q_\rho^{\lambda_M}(x_0,t_0))\leq \frac{\delta_1^{\frac{2(n+2)}{q}}\epsilon}{800},
\end{equation}
where $c_7=c_7(n,p,\nu,L)$ is the same constant as in Lemma \ref{lem:u-v}.
Let us briefly explain why we can choose such an $R_\epsilon'$. First, \eqref{eq5.8.3.0} and \eqref{eq5.8.3} are possible since $\omega$ is a nondecreasing function satisfying the Dini condition \eqref{dini}. Moreover, \eqref{eq5.8.4} and \eqref{eq5.8.5} are possible by using the assumption \eqref{asp1}. Finally, \eqref{eq5.8.6} is possible by Proposition \ref{lem:cty}, which is applicable since the assumption \eqref{asp1} directly implies \eqref{eq:mulim}.

Arguing exactly as in \eqref{eq5.2.6}, the bound \eqref{eq5.8.5} implies that
\begin{equation}\label{eq5.8.7}
     \sup_{(x_0,t_0)\in\tilde{Q}}\sup_{0<\rho\leq {R_\epsilon'} }\frac{|\mu|(Q_\rho^{\lambda_M}(x_0,t_0))}{\rho^{n+1}}\leq \frac{\delta_1^\frac{n+2}{q} \epsilon}{800c_7}\lambda_M.
\end{equation}
Now we fix $(x_0,t_0)\in \tilde{Q}$ and define a sequence of shrinking intrinsic cylinders for $i\in\mathbb{N}$, namely,
\begin{equation*}
    Q_i:=Q_{r_i}^{\lambda_M}(x_0,t_0), \quad r_i:=\delta_1^i {R_\epsilon'}.
\end{equation*}

\emph{Step 2: The iteration step.}
We have the following lemma.
\begin{lemma}\label{lem5.8.1}
Assume that
\begin{equation}\label{eq5.8.9}
   \Big(\fint_{Q_{i+1}} |Du|^q \,dxdt\Big)^{1/q}\geq\frac{\epsilon}{10}\lambda_M.
\end{equation}
Then we have
\begin{equation}\label{eq5.8.10}
   \phi_q(Du,Q_{i+1})\leq \frac{1}{4} \phi_q(Du,Q_{i})+c_7\delta_1^{-(n+2)/q} \Big[\frac{|\mu|(Q_{i}}{r_{i}^{n+1}}\Big]+c_7 \delta_1^{-(n+2)/q} \omega(r_{i})\lambda_M.
\end{equation}
\end{lemma}
\begin{proof}
Let $w$ and $v$ be the solutions defined in \eqref{eqa:w} and \eqref{eqa:v} respectively, with the choices $\rho=r_{i}$ and $\lambda=\lambda_M$. By \eqref{eq5.2.2} and \eqref{eq5.8.7}, we can apply Lemma \ref{lem:supinf} with choices of parameters
$\rho=r_{i}$, $\lambda=\lambda_M$, $\delta=\delta_1$, and  $\theta=\epsilon^q/20$. Thus using \eqref{eq5.8.9} we obtain
$$
\frac{\lambda_M}{B}=\frac{\epsilon\lambda_M}{400n}\leq \sup_{Q_{i+1}}\|Dv\|, \quad s+\sup_{\frac{1}{4}Q_{i}} \|Dv\|\leq c_8\lambda_M=A\lambda_M.
$$
Recalling that $\delta_1=\delta_\gamma/4$ and applying Theorem \ref{thm:iter2} in $\frac{1}{4}Q_{i}=\frac{1}{4}Q_{r_{i}}^{\lambda_M}$, we have
\begin{equation}\label{eq5.8.11}
    \phi_q(Dv,Q_{i+1})=\phi(Dv, \frac{\delta_\gamma}{4}Q_{i})\leq {2^{-4(n+3)}}\phi_q(Dv,\frac{1}{4}Q_{i}).
\end{equation}
By \eqref{eq5.2.2}, \eqref{eq5.8.3}, \eqref{eq5.8.7},  and \eqref{eq5.8.11}, we can apply Lemma \ref{lem:vtou} (with $\epsilon=1$) and get \eqref{eq5.8.10}.
\end{proof}
\emph{Step 3: Exit time argument.}
The main result we want to prove is as follows.
\begin{lemma}\label{lem:diff5}
It holds that
\begin{equation*}
    |\mathbf{m}(Du,Q_j)-\mathbf{m}(Du,Q_k)|< {\epsilon\lambda_M}, \quad \forall \;0\leq j\leq k.
\end{equation*}
\end{lemma}
\begin{proof}
For simplicity, we still denote  $\mathbf{m}_i:=\mathbf{m}(Du,Q_i)$ and $\Phi_i:=\phi_q(Du, Q_i)$ for $i\geq 0$.
We denote the set
\begin{equation}\label{def:setl}
L:=\Big\{i\in \mathbb{N}:\; \Big(\fint_{Q_i} |Du|^q \,dxdt\Big)^{1/q}<\frac{\epsilon}{10}\lambda_M\Big\}.
\end{equation}
We can assume $0\leq j < k$ and there are two different cases:
$$
L\cap \{j+1,\ldots,k\}=\emptyset, \quad \text{or} \quad L\cap \{j+1,\ldots,k\}\neq\emptyset.
$$
\emph{Case 1: $L\cap \{j+1,\ldots,k\}=\emptyset$.}
By the definition of the set $L$ in \eqref{def:setl}, we can apply Lemma \ref{lem5.8.1} for $i\in \{j,\ldots,k-1\}$ and obtain
\begin{equation}\label{eq5.8.16}
    \Phi_{i+1}\leq \frac{1}{4} \Phi_{i} +c_7\delta_1^{-(n+2)/q} \Big[\frac{|\mu|(Q_{i})}{r_{i}^{n+1}}\Big]+c_7 \delta_1^{-(n+2)/q} \omega(r_{i})\lambda_M.
\end{equation}
Summing up \eqref{eq5.8.16} for $i\in \{j,\ldots,k-1\}$, and using standard manipulations as in the proof of \eqref{sumA}, we have
\begin{equation}\label{eq5.8.18}
\begin{aligned}
    &\sum_{i=j}^{k} \Phi_i\leq 2\Phi_j+2c_7\delta_1^{-(n+2)/q}\sum_{i=j}^{k-1} \frac{|\mu|(Q_{i})}{r_{i}^{n+1}}+2c_7 \delta_1^{-(n+2)/q} \sum_{i=j}^{k-1}\omega(r_{i})\lambda_M\\
    &\leq \frac{\delta_1^{4(n+2)/q}\epsilon}{400}+2c_7\delta_1^{-(n+2)/q}\sum_{i=j}^{k-1} \frac{|\mu|(Q_{i})}{r_{i}^{n+1}}+2c_7 \delta_1^{-(n+2)/q} \sum_{i=j}^{k-1}\omega(r_{i})\lambda_M.
\end{aligned}
\end{equation}
Here we also used \eqref{eq5.8.6} in the last inequality.
Using \eqref{eq:int:sum} (with $\lambda=1$), \eqref{eq5.8.4}, and the fact that $Q_i\subset Q_{\lambda_M^{(2-p)/2} r_i}(x_0,t_0)$, we have
\begin{equation}\label{eq5.8.19}
\begin{aligned}
\sum_{i=0}^\infty \frac{|\mu|(Q_i)}{r_i^{n+1}}&\leq \sum_{i=0}^\infty \frac{|\mu|\big(Q_{\lambda_M^{(2-p)/2} r_i}(x_0,t_0)\big)}{r_i^{n+1}}\\
&\leq \lambda_M^{(n+1)(2-p)/2} \delta_1^{-(n+2)} \int_{0}^{2 \lambda_M^{(2-p)/2}{R_\epsilon'}} \frac{|\mu|(Q_\rho(x_0,t_0))}{\rho^{n+1}}\frac{d\rho}{\rho}
\\&\leq \frac{\delta_1^{3(n+2)/q} \epsilon}{800c_7}.
\end{aligned}
\end{equation}
Similarly, from \eqref{eq5.8.3.0} we have
\begin{equation}\label{eq5.8.20}
  \sum_{i=0}^\infty \omega(r_i)\lambda_M\leq \frac{\delta_1^{3(n+2)/q} \epsilon}{800c_7}.
\end{equation}
Combining \eqref{eq5.8.18}, \eqref{eq5.8.19}, and \eqref{eq5.8.20} we obtain
$$
 \sum_{i=j}^{k} \Phi_i \leq \frac{\delta_1^{2(n+2)/q}\epsilon}{100}.
$$
By \eqref{eq:mdiff}, the triangle inequality, and the previous inequality, it holds that
\begin{align*}
    |\mathbf{m}_k-\mathbf{m}_j|\leq \sum_{i=j}^{k-1} |\mathbf{m}_{i+1}-\mathbf{m}_i|\leq 4\delta_1^{-(n+2)/q} \sum_{i=j}^k \Phi_i\leq \frac{\epsilon \lambda_M}{25}.
\end{align*}
\emph{Case 2: $L\cap \{j+1,\ldots,k\}\neq\emptyset$.}
We prove in this case that
\begin{equation*}
    |\mathbf{m}_j|< \frac{\epsilon\lambda_M}{2} \quad \text{and} \quad |\mathbf{m}_k|< \frac{\epsilon\lambda_M}{2}.
\end{equation*}
We only give the proof of the former inequality and the proof for the latter is similar.
By the assumption that $L\cap \{j+1,\ldots,k\}\neq\emptyset$, we can define
$j':=\min\{l\in L:\, l\geq j+1\}$ and we have $j'\in L$. Thus by \eqref{eq:mbound} and the fact that $q\in(1/2,1)$, we obtain
\begin{equation}\label{bound:mj1}
  |\mathbf{m}_{j'}|\leq 4\Big(\fint_{Q_j}|Du(x,t)|^q \,dxdt\Big)^{1/q}< \frac{2\epsilon\lambda_M}{5}.
\end{equation}
There are two possibilities, namely, $j'=j+1$ or $j'>j+1$. First, we assume $j'=j+1$. Using \eqref{eq:mdiff} and \eqref{eq5.8.6}, we have
\begin{equation}\label{mj-mj1}
    |\mathbf{m}_{j}-\mathbf{m}_{j+1}|\leq 2\Phi_{j+1}+2 \delta_1^{-(n+2)/q} \Phi_{j}\leq \frac{\epsilon}{200}\leq \frac{\epsilon \lambda_M}{200}.
\end{equation}
Therefore, by the triangle inequality, \eqref{bound:mj1}, and \eqref{mj-mj1},
$$
|\mathbf{m}_j|
\le |\mathbf{m}_j-\mathbf{m}_{j+1}|
+|\mathbf{m}_{j'}|<\frac{\epsilon \lambda_M}{25}+\frac{2\epsilon \lambda_M}{5}< \frac{\epsilon\lambda_M}{2}.
$$
Otherwise, we have $j'>j+1$. Then by the definition of $j'$, we know that
$L\cap \{j+1,\ldots, j'-1\}=\emptyset$. Therefore we can apply Lemma \ref{lem5.8.1} for $i\in\{j,\ldots, j'-2\}$. From now on, we can argue exactly as in Case 1 to get
\begin{equation}\label{mj-mj-1}
|\mathbf{m}_j-\mathbf{m}_{j'-1}|\leq \frac{\epsilon \lambda_M}{25}.
\end{equation}
Again using \eqref{eq:mdiff} and \eqref{eq5.8.6}, we have
\begin{equation}\label{mj-1-mj}
    |\mathbf{m}_{j'-1}-\mathbf{m}_{j'}|\leq 2\Phi_{j'}+2 \delta_1^{-(n+2)/q} \Phi_{j'-1}\leq \frac{\epsilon}{200}\leq \frac{\epsilon \lambda_M}{200}.
\end{equation}
Therefore, by the triangle inequality, \eqref{bound:mj1}, \eqref{mj-mj-1}, and \eqref{mj-1-mj},
$$
|\mathbf{m}_j|
\le |\mathbf{m}_j-\mathbf{m}_{j'-1}|
+|\mathbf{m}_{j'-1}-\mathbf{m}_{j'}|
+|\mathbf{m}_{j'}|<\frac{\epsilon \lambda_M}{25}+\frac{\epsilon \lambda_M}{200}+\frac{2\epsilon\lambda_M}{5}< \frac{\epsilon\lambda_M}{2}.
$$
The proof of the lemma is now completed.
\end{proof}
\emph{Step 4: Conclusion.}
For any $\rho_1,\;\rho_2\in(0,{R_\epsilon'}]$, there exists two integers $j,\,k\geq 0$ such that
$$
\delta_1^{j+1} {R_\epsilon'} <\rho_1\leq \delta_1^j {R_\epsilon'} \quad \text{and} \quad \delta_1^{k+1} {R_\epsilon'} <\rho_2\leq \delta_1^k {R_\epsilon'}.
$$
By \eqref{eq:mdiff}, we have
\begin{align*}
  &|\mathbf{m}(Du,Q_{\rho_1}^{\lambda_M}(x_0,t_0))-\mathbf{m}(Du,Q_{j})|\\
  &\leq 2 \phi_q(Du,Q_{\rho_1}^{\lambda_M}(x_0,t_0))+2 \Big(\frac{|Q_j|}{|Q_{\rho_1}^{\lambda_M}(x_0,t_0)|}\Big)^{1/q} \phi_q(Du,Q_j) \leq \frac{\epsilon}{200}.
\end{align*}
Here we used \eqref{eq5.8.6} in the last line.
Similarly, it holds that
$$
|\mathbf{m}(Du,Q_{\rho_2}^{\lambda_M}(x_0,t_0))-\mathbf{m}(Du,Q_{k})|\leq \frac{\epsilon}{200}.
$$
Thus, the estimate \eqref{eq 5.8.1} follows by using Lemma \ref{lem:diff5}, the triangle inequality, and the last two inequalities. The proposition is proved.
\end{proof}

\begin{proof}[Proof of Theorem \ref{thm:cty}.]
For any  $(x_0,t_0)\in \tilde{Q}\cap \mathcal{L}$, where $\mathcal{L}\equiv \mathcal{L}_{\lambda_M}$ is the set of Lebesgue points of $Du$ defined in \eqref{Lebesgue},
arguing exactly as in \eqref{eq:m-du}, we know that
\begin{equation*}
\lim_{r\to 0} |\mathbf{m}(Q_{r}^{\lambda_M}(x_0,t_0)-Du(x_0,t_0)|=0.
\end{equation*}
Hence by Proposition \ref{lem:conv},
\begin{equation}\label{limit5.3}
\lim_{r\to 0} |\mathbf{m}(Q_{r}^{\lambda_M}(x_0,t_0)-Du(x_0,t_0)|=0  \quad \text{uniformly in } (x_0,t_0)\in \tilde{Q}\cap \mathcal{L}.
\end{equation}
Let $(x_1,t_1),\, (x_2,t_2)\in \tilde{Q}\cap \mathcal{L}$ and  $\rho:=|(x_1,t_1)-(x_2,t_2)|_{\text{par}}$. Without loss of generality, we assume that $t_1\leq t_2$. Therefore $Q_\rho^{\lambda_M}(x_1,t_1)\subset Q_{2\rho}^{\lambda_M}(x_2,t_2)$ and by \eqref{eq:mdiff} and the fact that $q\in(\frac{1}{2}, 1)$, we have
\begin{equation*}
\begin{aligned}
    &|\mathbf{m}(Du, Q_\rho^{\lambda_M}(x_1,t_1))-\mathbf{m}(Du, Q_{2\rho}^{\lambda_M}(x_2,t_2))|
    \\&\leq 2\phi_q(Du,Q_\rho^{\lambda_M}(x_1,t_1))+2^{2n+5} \phi_q(Du, Q_{2\rho}^{\lambda_M}(x_2,t_2)).
\end{aligned}
\end{equation*}
By the triangle inequality and the previous inequality, we obtain
\begin{equation}\label{eq5.8.28}
\begin{aligned}
    &|Du(x_1,t_1)-Du(x_2,t_2)|
    \\&\leq |Du(x_1,t_1)-\mathbf{m}(Du, Q_\rho^{\lambda_M}(x_1,t_1))|
    +|\mathbf{m}(Du, Q_{2\rho}^{\lambda_M}(x_2,t_2))-Du(x_2,t_2)|
    \\&\quad+|\mathbf{m}(Du, Q_\rho^{\lambda_M}(x_1,t_1))-\mathbf{m}(Du, Q_{2\rho}^{\lambda_M}(x_2,t_2))|.
     \\&\leq |Du(x_1,t_1)-\mathbf{m}(Du, Q_\rho^{\lambda_M}(x_1,t_1))|
    +|\mathbf{m}(Du, Q_{2\rho}^{\lambda_M}(x_2,t_2))-Du(x_2,t_2)|
    \\&\quad +2\phi_q(Du,Q_\rho^{\lambda_M}(x_1,t_1))+2^{2n+5} \phi_q(Du, Q_{2\rho}^{\lambda_M}(x_2,t_2)).
\end{aligned}
\end{equation}
Using \eqref{limit5.3}, \eqref{eq5.8.28}, and Proposition \ref{lem:cty}, it follows that for
$(x_1,t_1),\, (x_2,t_2)\in \tilde{Q}\cap \mathcal{L}$,
$$
|Du(x_1,t_1)-Du(x_2,t_2)|\to 0 \quad \text{when} \,\, \rho\equiv |(x_1,t_1)-(x_2,t_2)|_{\text{par}}\to 0.
$$
Since $|\tilde{Q}\backslash \mathcal{L}|=0$, we conclude that $Du$ is continuous in $\tilde{Q}$.
\end{proof}

\begin{proof}[Proof of Corollary \ref{thm1.3}.]
By \cite[Lemma 2.1]{kuusi2013mingione}, we know that the assumption \eqref{asp2} implies \eqref{asp1}. Therefore, Corollary \ref{thm1.3} is a direct consequence of Theorem \ref{thm:cty}.
\end{proof}

\begin{proof}[Proof of Corollary \ref{thm1.4}.]
Corollary \ref{thm1.4} follows directly from Theorem \ref{thm:cty} since the assumptions \eqref{asp3} and \eqref{asp4} directly imply \eqref{asp1}.
\end{proof}

\bibliographystyle{plain}

\begin{thebibliography}{10}

\bibitem{bogelein2008higher}
Verena B\"{o}gelein and Mikko Parviainen.
\newblock
Higher integrability for weak solutions of higher order degenerate parabolic systems.
\newblock {\em Ann. Acad. Sci. Fenn. Math.}, 33(2):387--412, 2008.


\bibitem{luis2003estimates}
Luis~A. Caffarelli and Qingbo Huang.
\newblock Estimates in the generalized {C}ampanato-{J}ohn-{N}irenberg
              spaces for fully nonlinear elliptic equations.
\newblock {\em Duke Math. J.}, 118(1):1--17, 2003.


\bibitem{choi2019gradient}
Jongkeun Choi and Hongjie Dong.
\newblock Gradient estimates for {S}tokes systems in domains.
\newblock {\em Dyn. Partial Differ. Equ.}, 16(1):1--24, 2019.

\bibitem{dibenedetto1993degenerate}
Emmanuele DiBenedetto.
\newblock
Degenerate parabolic equations.
\newblock
{\em Universitext.}
\newblock {\em Springer-Verlag, New York}, 1993.

\bibitem{dong2012gradient}
Hongjie Dong.
\newblock Gradient estimates for parabolic and elliptic systems from linear laminates.
\newblock {\em Arch. Ration. Mech. Anal.}, 205(1):119--149, 2012.

\bibitem{dong2017c1}
Hongjie Dong and Seick Kim.
\newblock On {$C^1$}, {$C^2$}, and weak type-{$(1,1)$} estimates for linear
  elliptic operators.
\newblock {\em Comm. Partial Differential Equations}, 42(3):417--435, 2017.

\bibitem{dong2020on}
Hongjie Dong, Jihoon Lee, and Seick Kim.
\newblock On conormal and oblique derivative problem for elliptic equations with {D}ini mean oscillation coefficients.
\newblock {\em Indiana Univ. Math. J.},
69(6):1815--1853, 2020.

\bibitem{dong2021gradient}
Hongjie Dong and Hanye Zhu.
\newblock
Gradient estimates for singular $p$-Laplace type equations with measure data.
\newblock {\em arXiv preprint arXiv:2102.08584}, 2021.

\bibitem{MR2729305}
Frank Duzaar and Giuseppe Mingione.
\newblock Gradient continuity estimates.
\newblock {\em Calc. Var. Partial Differential Equations}, 39(3-4):379--418,
  2010.

\bibitem{duzaar2010gradient}
Frank Duzaar and Giuseppe Mingione.
\newblock Gradient estimates via linear and nonlinear potentials.
\newblock {\em J. Funct. Anal.}, 259(11):2961--2998, 2010.

\bibitem{MR2823872}
Frank Duzaar and Giuseppe Mingione.
\newblock Gradient estimates via non-linear potentials.
\newblock {\em Amer. J. Math.}, 133(4):1093--1149, 2011.

\bibitem{giusti2003direct}
Enrico Giusti.
\newblock  Direct methods in the calculus of variations.
\newblock {\em World Scientific Publishing Co., Inc., River Edge, NJ}, 2003.

\bibitem{krylov2010on}
Nicolai~V. Krylov.
\newblock On Bellman's equations with VMO coefficients.
\newblock {\em Methods Appl. Anal.}, 17(1):105--121, 2010.


\bibitem{kuusi2012new}
Tuomo Kuusi and Giuseppe Mingione.
\newblock
New perturbation methods for nonlinear parabolic problems.
\newblock {\em J. Math. Pures Appl. (9)}, 98(4):390--427, 2012.

\bibitem{kuusi2013mingione}
Tuomo Kuusi and Giuseppe Mingione.
\newblock Gradient regularity for nonlinear parabolic equations.
\newblock {\em Ann. Sc. Norm. Super. Pisa Cl. Sci. (5)}, 12(4):755--822, 2013.

\bibitem{kuusi2014guide}
Tuomo Kuusi and Giuseppe Mingione.
\newblock Guide to nonlinear potential estimates.
\newblock {\em Bull. Math. Sci.}, 4(1):1--82, 2014.

\bibitem{MR3004772}
Tuomo Kuusi and Giuseppe Mingione.
\newblock Linear potentials in nonlinear potential theory.
\newblock {\em Arch. Ration. Mech. Anal.}, 207(1):215--246, 2013.

\bibitem{kuusi2014riesz}
Tuomo Kuusi and Giuseppe Mingione.
\newblock
 Riesz potentials and nonlinear parabolic equations.
\newblock {\em Arch. Ration. Mech. Anal.}, 212(3):727--780, 2014.

\bibitem{kuusi2014the}
Tuomo Kuusi and Giuseppe Mingione.
\newblock
 The {W}olff gradient bound for degenerate parabolic equations.
\newblock {\em J. Eur. Math. Soc. (JEMS)}, 16(4):835--892, 2014.

\bibitem{MR1233190}
Gary~M. Lieberman.
\newblock Sharp forms of estimates for subsolutions and supersolutions of
  quasilinear elliptic equations involving measures.
\newblock {\em Comm. Partial Differential Equations}, 18(7-8):1191--1212, 1993.

\bibitem{lieberman1996second}
Gary~M. Lieberman.
\newblock Second order parabolic differential equations.
\newblock {\em World Scientific Publishing Co., Inc., River Edge, NJ}, 1996.

\bibitem{MR2746772}
Giuseppe Mingione.
\newblock Gradient potential estimates.
\newblock {\em J. Eur. Math. Soc. (JEMS)}, 13(2):459--486, 2011.

\bibitem{nguyen2019good}
Quoc-Hung Nguyen and Nguyen~Cong Phuc.
\newblock Good-{$\lambda$} and {M}uckenhoupt-{W}heeden type bounds in
  quasilinear measure datum problems, with applications.
\newblock {\em Math. Ann.}, 374(1-2):67--98, 2019.

\bibitem{nguyen2020existence}
Quoc-Hung Nguyen and Nguyen~Cong Phuc.
\newblock Existence and regularity estimates for quasilinear equations with
  measure data: the case $1< p\leq \frac{3n- 2}{2n-1}$.
\newblock {\em arXiv preprint arXiv:2003.03725}, 2020.

\bibitem{nguyen2020pointwise}
Quoc-Hung Nguyen and Nguyen~Cong Phuc.
\newblock Pointwise gradient estimates for a class of singular quasilinear
  equations with measure data.
\newblock {\em J. Funct. Anal.}, 278(5):108391, 35, 2020.

\bibitem{park2020regularity}
Jung-Tae Park and Pilsoo Shin.
\newblock
Regularity estimates for singular parabolic measure data problems with sharp growth.
\newblock {\em arXiv preprint arXiv:2004.03889}, 2020.





\end{thebibliography}

\end{document}